\documentclass[11pt,final]{amsart}
\usepackage[utf8]{inputenc}
\usepackage[T1]{fontenc}
\usepackage[letterpaper,centering]{geometry}
\usepackage{lmodern}
\usepackage{eucal}
\usepackage{mathrsfs}
\usepackage{enumerate}
\usepackage{amsmath,amssymb,amsfonts}
\usepackage{xcolor}
\definecolor{darkblue}{rgb}{0,0,0.3}
\definecolor{darkgreen}{rgb}{0,0.4,0}
\usepackage[colorlinks=true,linkcolor=darkblue, citecolor=darkgreen, unicode,pdfborder={0 0 0}]{hyperref}
\usepackage[ps,arrow,matrix,tips,line]{xy}
{\setbox0\hbox{$ $}}\fontdimen16\textfont2=\fontdimen17\textfont2
\entrymodifiers={+!!<0pt,\the\fontdimen22\textfont2>}
\SelectTips{cm}{11}

\usepackage{ifdraft}
\ifdraft{
\usepackage[notcite,notref]{showkeys}
\addtolength{\hoffset}{1.5cm}

}{}

\theoremstyle{plain}
\newtheorem{thm}[equation]{Theorem}
\newtheorem{conj}[equation]{Conjecture}
\newtheorem{problem}[equation]{Problem}

\newtheorem{lem}[equation]{Lemma}
\newtheorem{cor}[equation]{Corollary}
\newtheorem{prop}[equation]{Proposition}
\theoremstyle{definition}

\newtheorem{defn}[equation]{Definition}
\newtheorem{rmk}[equation]{Remark}
\newtheorem{rmks}[equation]{Remarks}
\newtheorem{example}[equation]{Example}

\numberwithin{equation}{subsection}

\let\oldsubsubsection\subsubsection
\renewcommand{\subsubsection}[1]{%
\setcounter{subsubsection}{\value{equation}}%
\oldsubsubsection{#1}%
\refstepcounter{equation}%
}

\input cyracc.def
\DeclareFontFamily{U}{russian}{}
\DeclareFontShape{U}{russian}{m}{n}
        { <5><6> wncyr5
        <7><8><9> wncyr7
        <10><10.95><12><14.4><17.28><20.74><24.88> wncyr10 }{}
\DeclareSymbolFont{Russian}{U}{russian}{m}{n}
\DeclareSymbolFontAlphabet{\mathcyr}{Russian}
\makeatletter
\let\@math@cyr\mathcyr
\renewcommand{\mathcyr}[1]{\@math@cyr{\cyracc #1}}
\makeatother
\newcommand{\geom}{{\mathrm{geom}}}
\renewcommand{\ni}{{\mathrm{ni}}}

\newcommand{\isoto}{\myxrightarrow{\,\sim\,}}
\makeatletter
\def\myrightarrow{{\setbox\z@\hbox{$\rightarrow$}\dimen0\ht\z@\multiply\dimen0 6\divide\dimen0 10\ht\z@\dimen0\box\z@}}
\def\myrightarrowfill@{\arrowfill@\relbar\relbar\myrightarrow}
\newcommand{\myxrightarrow}[2][]{\ext@arrow 0359\myrightarrowfill@{#1}{#2}}
\makeatother
\newcommand{\mtilde}{{\mathchoice
    {\widetilde{m}}
    {\widetilde{m}}
    {\rlap{$\scriptscriptstyle{m}$}\vphantom{\raise0pt\hbox{$m$}}\smash{\lower2.5pt\hbox{$\scriptscriptstyle\widetilde{\phantom{\scriptscriptstyle{m}}}$}}}
    {\rlap{$\scriptscriptstyle{m}$}\vphantom{\raise.2pt\hbox{$m$}}\smash{\lower2.05pt\hbox{$\scriptscriptstyle\widetilde{\phantom{\scriptscriptstyle{m}}}$}}}}}

\newcommand{\Mtilde}{{\mathchoice
    {\rlap{$M$}\mkern1mu\smash[b]{\lower.5pt\hbox{$\widetilde{\phantom{M}}$}}\mkern-1mu}
    {\rlap{$M$}\mkern1mu\smash[b]{\lower.5pt\hbox{$\widetilde{\phantom{M}}$}}\mkern-1mu}
    {\rlap{$\scriptstyle{M}$}\mkern1mu\smash[b]{\lower.5pt\hbox{$\widetilde{\phantom{\scriptstyle{M}}}$}}\mkern-1mu}
    {\widetilde{M}}}}

\newcommand{\et}{{\mathrm{\acute et}}}
\newcommand{\sC}{{\mathscr C}}

\newcommand{\sH}{{\mathscr H}}

\newcommand{\sU}{{\mathscr U}}
\newcommand{\sV}{{\mathscr V}}

\newcommand{\A}{{\mathbf A}}

\newcommand{\F}{{\mathbf F}}
\newcommand{\Fp}{{\mathbf F}_{\mkern-2mup}}

\renewcommand{\P}{{\mathbf P}}
\newcommand{\Q}{{\mathbf Q}}
\newcommand{\R}{{\mathbf R}}
\newcommand{\C}{{\mathbf C}}

\newcommand{\Z}{{\mathbf Z}}

\newcommand{\Qp}{{\Q_{p}}}

\newcommand{\nr}{\mathrm{nr}}

\newcommand{\Gm}{\mathbf{G}_\mathrm{m}}
\newcommand{\Gmbark}{\mathbf{G}_{\mathrm{m},\bar k}}

\newcommand{\Gal}{\mathrm{Gal}}
\newcommand{\SL}{\mathrm{SL}}
\newcommand{\GL}{\mathrm{GL}}
\newcommand{\Cl}{\mathrm{Cl}}
\newcommand{\PSL}{\mathrm{PSL}}
\newcommand{\Pic}{\mathrm{Pic}}
\newcommand{\Br}{\mathrm{Br}}
\newcommand{\Spec}{\mathrm{Spec}}

\renewcommand{\phi}{\varphi}
\renewcommand{\emptyset}{\varnothing}

\newcommand{\Ker}{{\mathrm{Ker}}}

\newcommand{\Hom}{{\mathrm{Hom}}}

\newcommand{\mmu}{\boldsymbol{\mu}}

\newcommand{\Zhat}{{\hat\Z}}
\newcommand{\chapeau}{{\rlap{\smash{\hbox{\lower4pt\hbox{\hskip1pt$\widehat{\phantom{u}}$}}}}}}
\newcommand{\Picplushat}{\Pic_+^{{\smash{\hbox{\lower4pt\hbox{\hskip0.4pt$\widehat{\phantom{u}}$}}}}}}
\newcommand{\PicplusAhat}{\Pic_{+,\A}^{{\smash{\hbox{\lower4pt\hbox{\hskip.4pt$\widehat{\phantom{u}}$}}}}}}

\newcommand{\Pichat}{\Pic^{{\smash{\hbox{\lower4pt\hbox{\hskip0.4pt$\widehat{\phantom{u}}$}}}}}}

\DeclareMathOperator{\inv}{inv}

 \makeatletter
 \renewcommand{\tocsection}[3]{%
   \indentlabel{\@ifnotempty{#2}{\bfseries\ignorespaces#1 #2\quad}}\bfseries#3}

 \renewcommand{\tocsubsection}[3]{%
   \indentlabel{\@ifnotempty{#2}{\hspace{1.6em}\ignorespaces#1 #2\quad}}#3}
 \makeatother

\makeatletter\let\@wraptoccontribs\wraptoccontribs\makeatother

\date{February 14th, 2023; revised on October 3rd, 2023.}
\title[Around the inverse Galois problem]{Park City lecture notes:\\around the inverse Galois problem}

\author{Olivier Wittenberg}
\address{Institut Galil\'ee, Universit\'e Sorbonne Paris Nord, 99~avenue Jean-Baptiste Cl\'ement, 93430 Villetaneuse, France}
\email{wittenberg@math.univ-paris13.fr}

\begin{document}

\begin{abstract}
The inverse Galois problem asks whether any finite group can be realised as
the Galois group of a Galois extension of the rationals. This problem and
its refinements have stimulated a large amount of research in number theory
and algebraic geometry in the past century, ranging from Noether's problem
(letting~$X$ denote the quotient of the affine space by a finite group acting
linearly, when is~$X$ rational?)\ to the rigidity method (if~$X$ is not
rational, does it at least contain interesting rational curves?)\ and to
the arithmetic of unirational varieties (if all else fails, does~$X$ at least
contain interesting rational points?). The goal of the present notes, which formed the basis for three lectures
given at the Park City Mathematics Institute in August~2022,
is to provide an introduction to these topics.
\end{abstract}

\maketitle

The inverse Galois problem is a
simple-looking but fundamental open question of number theory on which
tools coming from diverse areas of mathematics can be brought to bear.
These lectures aim to explain the problem as well as a few of the many methods
that have been developed to attack it,
emphasising a geometric point of view whenever possible.

The first lecture introduces the problem together with a refinement of it first considered by Grunwald,
and presents the strategy of Hilbert and Noether---a strategy
based on Hilbert's irreducibility theorem
and laid out
more than a hundred years ago.
This leads us to
 the notion of versal torsor,
and to
 questions of rationality, stable rationality, retract rationality
for quotient varieties.
The second lecture is devoted to the regular inverse Galois problem, which is about the construction
of Galois covers of curves.  We present the rigidity method in detail,
via Hurwitz moduli spaces.
The third and final lecture explains how Grunwald's problem is expected to be
controlled by the Brauer--Manin obstruction to weak approximation on the quotient varieties appearing in the
Hilbert--Noether strategy, and discusses the descent method and its applications to the inverse Galois problem
and to its variants.

Several important topics could not fit into these three lectures
and had to be left aside, such as the connection between the inverse Galois problem
and the construction of Galois representations (see \cite{shih}, \cite{zywina} for some examples),
the directions in which the rigidity method has been developed beyond its base case
(see e.g.\ \cite{dettweilerreiter}, \cite{volkleinbc}, \cite[Chapter~III]{mallematzat} among many others),
or the study of embedding problems.

Additional material on the inverse Galois problem can be found
in \cite{serretopics,mallematzat,volklein,debessurvey,jensenledetyui,matzatkonstruktive,szamuelygaloisgroups}.

\bigskip
\emph{Acknowledgements.}
These lectures were delivered in August~2022 in Park City on the occasion of the PCMI~2022 graduate summer school,
whose organisers
are gratefully acknowledged.
The author thanks Diego Izquierdo for leading the problem sessions,
Enric Florit
and Boris Kunyavski\u{\i}
for correcting  oversights in the first version of the notes,
as well as Jean-Louis Colliot-Thélène, Joachim König, Jean-Pierre Serre, Tamás Szamuely and the referee for numerous useful comments that helped improve the exposition.

\section{From Galois to Hilbert and Noether}
\label{sec:1}

\subsection{Introduction}

Galois theory turns the collection of all number fields into a profinite
group, the absolute Galois group $\Gal(\bar \Q/\Q)$ of~$\Q$.
The study of this group and of
its representations has been a cornerstone of number theory for more than a century.
Yet, even such a basic question as the following one remains wide-open to this day:
do all finite groups appear as quotients of $\Gal(\bar \Q/\Q)$?
This is the so-called ``inverse Galois problem''.

The same question can be asked about the absolute Galois group of an
arbitrary field~$k$: given a finite group~$G$ and a field~$k$,
does there exist a Galois field extension~$K$ of~$k$ such that $\Gal(K/k)\simeq G$?
Obviously, the answer is in the negative for some fields~$k$ that have a small absolute Galois group
(e.g.\ the fields $\C$ and $\R$, trivially; or the field $\Q_p$, as its absolute Galois group is prosolvable).
When~$k$ is a number field, a positive answer is known when~$G$
is solvable (Shafarevich, see \cite[Chapter~IX, \textsection6]{neukirchschmidtwingberg}
and the references therein), when~$G$ is a symmetric or an alternating group (Hilbert~\cite{hilbertorig}),
when~$G$ is a sporadic group other than~$M_{23}$
(Shih,
Fried,
Belyi,
Matzat,
Thompson,
Hoyden--Siedersleben,
Zeh--Marschke,
Hunt,
Pahlings, see \cite{mallematzat}),
when~$G$ belongs to various infinite families of non-abelian simple groups of Lie type (e.g.\ the groups $\PSL_2(\Fp)$, according to Shih, Malle, Clark, Zywina; see \cite{zywina});
but the problem remains open over~$\Q$ even for such a small group as
 $\SL_2(\F_{13})$,
of order~$2184$,
or as the simple group
 $\PSL_3(\F_{8})$
(for the latter, see~\cite{zywinasmall}, to be complemented with \cite{dieulefaitfloritvila}).

Several variants or generalisations of the inverse Galois problem have been considered in the literature.
Here is one of them.  Given a number field~$k$, we denote the set of its places by~$\Omega$
and the completion of~$k$ at $v \in \Omega$ by~$k_v$.

\begin{problem}[Grunwald]
\label{pb:grunwald}
Let~$k$ be a number field
and $S\subset \Omega$ be a finite subset.
Let~$G$ be a finite group.
For each $v \in S$, let~$K_v$ be a Galois field extension of~$k_v$
such that the group $\Gal(K_v/k_v)$ can be embedded into~$G$.
Does there exist a Galois field extension~$K$ of~$k$ such that $\Gal(K/k)\simeq G$
and such that for all $v \in S$, the completion of~$K/k$ at any place of~$K$ lying above~$v$
is isomorphic to $K_v/k_v$?
\end{problem}

The Grunwald--Wang theorem,
which was proved by Wang~\cite{wanggrunwald} following the work of Grunwald~\cite{grunwald}
and which has an interesting history (see
\cite[Chapter~X, footnote on p.~73]{artintate} and
\cite[Chapter~VIII, §2, p.~234, Notes]{milneCFT}),
 gives a complete answer
to Problem~\ref{pb:grunwald}
 when~$G$ is abelian,
via class field theory.
In particular, the answer to Grunwald's problem is negative  for $G=\Z/8\Z$ and $k=\Q$
(see Proposition~\ref{prop:wang} below),
but it is positive, for any number field~$k$ and any finite abelian group~$G$,
 as soon as~$S$ does not contain
any place dividing~$2$.
For an arbitrary finite group~$G$,  the Grunwald problem
is expected to have a positive answer whenever~$G$ does not contain any place dividing
the order of~$G$.  This is the ``tame'' Grunwald problem,
a terminology coined in \cite{dlan}.

Other variants include embedding problems (given a Galois field extension~$\ell/k$,
a finite group~$G$
and
a surjection $\phi:G \twoheadrightarrow \Gal(\ell/k)$, can one embed~$\ell/k$ into a Galois
field extension~$K/k$ such that $G\simeq \Gal(K/k)$, the map~$\phi$
being identified with the restriction map
$\Gal(K/k) \twoheadrightarrow \Gal(\ell/k)$?)\ or
the question of resolving the inverse Galois problem with
additional constraints, such as the constraint
that a given finite collection of elements of~$k$ be norms from~$K$
(a problem raised in \cite{freiloughrannewton}).

\subsection{Torsors and Galois extensions}
\label{subsec:torsorsandgalois}

Let us start by reformulating the inverse Galois problem in terms of torsors.
Hereafter, a \emph{variety} over a field~$k$ is a separated scheme of finite type over~$k$
(which may be disconnected or otherwise reducible) and~$\bar k$ denotes an algebraic closure of~$k$.

\begin{defn}
\label{def:torsor}
Let $\pi:Y\to X$ be a finite morphism between varieties over a field~$k$.
Let~$G$ be a finite group acting on~$Y$, in such a way that~$\pi$ is $G$\nobreakdash-equivariant
(for the trivial action of~$G$
on~$X$).
We say that~$\pi$ is a \emph{$G$\nobreakdash-torsor},
or that~$Y$ is a $G$\nobreakdash-torsor over~$X$,
if~$\pi$ is étale and~$G$
acts simply transitively on the fibres of the map $Y(\bar k)\to X(\bar k)$ induced by~$\pi$.
\end{defn}

When~$G$ is a finite group acting on a variety~$Y$, we denote by $Y/G$ the quotient variety,
characterised by the universal property of quotients,
when it exists.
Let us recall that the quotient~$Y/G$ exists if~$Y$ is quasi-projective;
the projection
$\pi:Y \to Y/G$ is then finite and surjective; it is étale if the action of~$G$
on~$Y$ is free (by which we mean
that~$G$
acts freely on the set~$Y(\bar k)$); and in the affine case, if $Y=\Spec(A)$, then $Y/G=\Spec(A^G)$ (see
\cite[Chapter~II, §7 and Chapter~III, §12]{mumford}).

It is easy to see that a finite $G$\nobreakdash-equivariant morphism $\pi:Y\to X$
is a $G$\nobreakdash-torsor if and only if~$G$ acts freely on~$Y$
and~$\pi$ induces an isomorphism
$Y/G \isoto X$.
Thus, in particular, a Galois field extension~$K/k$ with Galois group~$G$
is the same thing
as an irreducible $G$\nobreakdash-torsor over~$k$ (that is, over~$\Spec(k)$);
and this, in turn, is the same thing as an irreducible variety of dimension~$0$, over~$k$,
endowed with a simply transitive
action of~$G$.

This rewording leads to a slight change in perspective, first emphasised by
Hilbert and Noether:
in order to solve the inverse Galois problem for~$G$,
we can now start with any irreducible quasi-projective variety~$Y$ endowed with a free action of~$G$;
setting $X=Y/G$,  we  obtain a
 $G$\nobreakdash-torsor $\pi:Y\to X$; it is then enough
to look for rational points $x \in X(k)$ such that the fibre~$\pi^{-1}(x)$
is irreducible.  Indeed, this fibre is in any case a $G$\nobreakdash-torsor over~$k$.

\begin{rmk}
\label{rmk:pushtorsor}
Given a subgroup $H \subseteq G$, any $H$\nobreakdash-torsor $Y\to X$
gives rise to a
 $G$\nobreakdash-torsor
$Y'\to X$.
Namely, if~$H$ acts on the left on~$Y$,
we let it act on the right on
$G \times Y$ by $(g,y) \cdot h=(gh,h^{-1}y)$
and observe that $Y'=(G \times Y)/H$ inherits a free left action of~$G$,
turning the projection $Y' \to X$
into a $G$\nobreakdash-torsor.
The variety~$Y'$ is
(canonically) a disjoint union indexed by~$G/H$ of varieties
each of which is (non-canonically, in general) isomorphic, over~$X$, to~$Y$.
Conversely, if $Y' \to X$ is a $G$\nobreakdash-torsor and~$X$ is connected,
then any connected component~$Y$ of~$Y'$ is an $H$\nobreakdash-torsor over~$X$
for some subgroup~$H$ (namely, the stabiliser of~$Y$), and $Y'$ coincides with $(G \times Y)/H$.
All in all, when~$X$ is connected, the data of a $G$\nobreakdash-torsor $Y'\to X$ together with the choice
of a connected component of~$Y'$ is equivalent to the data of a subgroup $H \subseteq G$ and of a
connected $H$\nobreakdash-torsor $Y\to X$.
In particular, we see that if~$G$ is a finite group and~$k$ is a field,
the data of a $G$\nobreakdash-torsor over~$k$ together with the choice of a connected component
is equivalent to the data of a Galois field extension~$K/k$ endowed with an embedding $\Gal(K/k) \hookrightarrow G$.
\end{rmk}

\subsection{Hilbert's irreducibility theorem}
When the base of the $G$\nobreakdash-torsor $\pi:Y\to X$
is an open subset of~$\P^1_k$, with~$k$ a number field, and its total space~$Y$ is irreducible, the existence of
rational points $x \in X(k)$
such that the fibre $\pi^{-1}(x)$
is irreducible
is guaranteed by Hilbert's irreducibility theorem, which we state next.

\begin{thm}[Hilbert]
\label{thm:irredhilbertP1}
Let~$k$ be a number field.
Let $X \subseteq \P^1_k$ be a dense open subset.  Let $\pi: Y \to X$ be an irreducible étale covering
(i.e.\ a finite étale morphism from an irreducible variety).
There exists $x \in X(k)$ such that $\pi^{-1}(x)$ is irreducible.
\end{thm}

Theorem~\ref{thm:irredhilbertP1}
is classically
formulated in the following equivalent way: given an irreducible two-variable polynomial $f(s,t)$
with coefficients in a number field~$k$, there exist infinitely many
$t_0 \in k$ such that $f(s,t_0)$ is an irreducible one-variable polynomial with coefficients in~$k$.
In fact, the set of such~$t_0$
is not just infinite:
asymptotically, it contains $100\%$ of the elements of~$k$, when they are ordered by height
(see \cite[\textsection13.1, Theorem~3]{serremw}).

A proof of Theorem~\ref{thm:irredhilbertP1}
can be found in \cite[\textsection9.2, \textsection9.6]{serremw}, where the next corollary
is also established.

\begin{cor}
\label{cor:irredhilbertPn}
Same statement,
with~$X$ now a dense open subset of~$\P^n_k$ for some $n\geq 1$.
\end{cor}

Combining Corollary~\ref{cor:irredhilbertPn} with 
the remarks of~\textsection\ref{subsec:torsorsandgalois} leads to an observation, originating from Hilbert's work,
which is
extremely effective for the inverse Galois problem.  Before stating it in Corollary~\ref{cor:cortohilbert} below,
let us
recall that a variety~$X$ over a field~$k$ is said to be \emph{rational}
if it is birationally equivalent to an affine space, i.e.\ if it contains a dense open subset isomorphic
to a dense open subset of an affine space; when~$X$ is irreducible and reduced, this
means that its function field~$k(X)$ is a purely transcendental extension of~$k$.

\begin{cor}
\label{cor:cortohilbert}
Let~$k$ be a number field. Let~$G$ be a finite group.
If there exist an irreducible quasi-projective variety~$Y$ over~$k$ and a faithful action of~$G$ on~$Y$ such that
the quotient $Y/G$ is rational, then the inverse Galois problem admits a positive
solution for~$G$ over~$k$.
\end{cor}

\begin{proof}
As~$G$ acts faithfully on~$Y$, it acts freely on a dense open subset of~$Y$, say~$V$.  Indeed,
for $g \in G$, the locus in~$Y$ fixed by~$g$ is a closed subset of~$Y$ of codimension~$\geq 1$;
one can take for~$V$ the complement of the (finite) union of these fixed loci.
After shrinking~$V$,
we may assume that $V/G$ is isomorphic to an open subset of~$\P^n_k$.
Corollary~\ref{cor:irredhilbertPn} can now be applied to the projection $V\to V/G$.
\end{proof}

\begin{example}
The order~$3$ automorphism of~$\P^1_k$ given, in homogeneous coordinates,
by $[x:y] \mapsto [y:y-x]$ induces a faithful action of $G=\Z/3\Z$ on~$\P^1_k$.
The quotient $\P^1_k/G$ is rational since it is a unirational curve (Lüroth's theorem).
Thus, any number field admits a cyclic extension of degree~$3$.
\end{example}

Ekedahl proved the following useful generalisation of Hilbert's irreducibility theorem.
We recall that~$\Omega$ denotes the set of places of a number field~$k$.

\begin{thm}[\cite{ekedahl}]
\label{th:ekedahl}
Let $\pi:Y \to X$ be a finite étale morphism between geometrically irreducible varieties
over a number field~$k$.
Let $S \subset \Omega$ be a finite subset.
If $X(k)\neq\emptyset$,
then there exists a nonempty open
subset $\sU \subset \prod_{v \in \Omega\setminus S} X(k_v)$
 such that for all $x \in X(k) \cap \sU$,
the fibre $\pi^{-1}(x)$ is irreducible.
\end{thm}

In the above statement, we view~$X(k)$ as diagonally embedded
into $\prod_{v \in \Omega \setminus S}X(k_v)$,
which we endow with the product of the $v$\nobreakdash-adic topologies.
When~$X$ is rational, the set $X(k) \cap \sU$ is automatically nonempty,
by the weak approximation property for affine spaces. In general, though, this set can be empty.
A proof of Theorem~\ref{th:ekedahl}, at least in the Galois case\footnote{It
can be checked that Theorem~\ref{th:ekedahl}, in the Galois case,
still holds when~$Y$ is only assumed to be irreducible, instead of geometrically irreducible.  Under this weaker
assumption on~$Y$,
the Galois case does imply the general case; hence the validity of Theorem~\ref{th:ekedahl} as stated (and even slightly
more generally than stated,
since this weaker assumption on~$Y$ also suffices in the non-Galois case).}, which is the only one that we shall use
(we use it in the proof of Proposition~\ref{prop:versalrationalimpliesgrunwald} below),
can be found in \cite[Theorem~1.3]{ekedahl}.

\subsection{Noether's problem: statement}

The following problem, raised by Emmy Noether, takes on particular importance
in view of Corollary~\ref{cor:cortohilbert}.

\begin{problem}[Noether]
\label{pb:noether}
Let~$G$ be a finite group and~$k$ be a field.
Choose an embedding $G \hookrightarrow S_n$ for some $n\geq 1$.
Let~$G$ act on~$\A^n_k$ through
this embedding by permuting the
coordinates.  Is the quotient $\A^n_k/G$ rational over~$k$?
\end{problem}

By Corollary~\ref{cor:cortohilbert}, when~$k$ is a number field, a positive answer to Noether's problem for~$G$ over~$k$
implies a positive answer
to the inverse Galois problem for~$G$ over~$k$.
Beyond this implication,
Noether's problem has become a central problem in the study of rationality,
and has been the focus of
much research for its own sake.

\begin{example}
\label{ex:noethersn}
Noether's problem has a positive answer, over any field, for the symmetric group $G=S_n$.
Indeed, for $G=S_n$, the quotient $\A^n_k/G$ is in fact isomorphic
to $\A^n_k$, as the ring
$k[x_1,\dots,x_n]^{S_n}$
 of symmetric polynomials coincides with the polynomial ring in the elementary symmetric
polynomials.  Thus, in particular,
every number field admits a Galois field extension with group~$S_n$, for every $n\geq 1$.
\end{example}

\begin{example}
Noether's problem has a positive answer, for any $n \geq 1$ and any embedding $G \hookrightarrow S_n$, when~$G$ is
an abelian group of exponent~$e$ and~$k$ is
a field
that contains the $e$th
roots of unity
and
 whose characteristic does not divide~$e$.
In particular, it has a positive answer for all abelian groups over~$\C$.
This is a theorem of Fischer~\cite{fischer}.
\end{example}

\begin{example}
\label{ex:noethera5}
Noether's problem has a positive answer, over any field, for the
alternating group $G=A_5$.  This is a theorem of Maeda~\cite{maeda}.
On the other hand, as soon as $n \geq 6$, Noether's problem for $G=A_n$ is still open, over any field.
\end{example}

Noether knew that her problem has a positive answer for small groups
(namely, for all subgroups of~$S_4$).
In general, however, its answer is often negative,
as we discuss in some detail in~\textsection\ref{subsec:noether:counter} below.

\subsection{Versal torsors}
For some $G$\nobreakdash-torsors $\pi:Y\to X$, the existence of rational points $x \in X(k)$
such that $\pi^{-1}(x)$ is irreducible is not only a sufficient condition for a positive answer
to the inverse Galois problem for~$G$ over~$k$, but it is also necessary.
These are the \emph{versal} torsors.

\begin{defn}
Let~$G$ be a finite group, let~$k$ be a field and let~$X$ be a variety over~$k$.
A~$G$\nobreakdash-torsor $\pi:Y\to X$ is \emph{weakly versal}
if for any field extension $k'/k$ with~$k'$ infinite,
every $G$\nobreakdash-torsor over~$k'$ can be realised as the fibre of~$\pi$ above a $k'$\nobreakdash-point
of~$X$.
It is \emph{versal} if for any dense open subset $U \subseteq X$, the induced $G$\nobreakdash-torsor
$\pi^{-1}(U) \to U$ is weakly versal.
\end{defn}

\begin{example}
\label{ex:noetherisversal}
Choose an embedding $G \hookrightarrow S_n$ for some $n\geq 1$ and let~$G$ act on~$\A^n_k$ through
this embedding by permuting the
coordinates.  Let $Y$ be the open subset of~$\A^n_k$ consisting of the points whose coordinates are all
pairwise distinct.
Then~$G$ acts freely on~$Y$ and
it can be checked,
as a consequence of Hilbert's\footnote{Despite its name Hilbert's Theorem~90,
this theorem, for arbitrary~$n$, is due to Speiser~\cite{speiser}.} Theorem~90,
 according to which the Galois cohomology set $H^1(k',\GL_n)$
is a singleton for any field~$k'$,
that
 the resulting torsor $\pi:Y\to X=Y/G$ is versal (see \cite[Example~5.5]{garibaldimerkurjevserre}).
\end{example}

\begin{example}
\label{ex:noetherisversalsln}
Choose an embedding $G \hookrightarrow \SL_n(k)$ for some $n\geq 1$ and let~$G$ act on the algebraic
group $\SL_n$ over~$k$ through this embedding by right multiplication.
This action is free and it can again be checked, as a consequence of Hilbert's Theorem~90,
that the resulting torsor $\pi:\SL_n\to \SL_n/G$ is versal.
The argument for this is the same as in \cite[Example~5.5]{garibaldimerkurjevserre} once one knows that
the Galois cohomology set $H^1(k',\SL_n)$
is a singleton for any field~$k'$; the latter fact easily follows from Hilbert's Theorem~90.
\end{example}

\begin{rmk}
\label{rmk:noname}
Two varieties~$V$ and~$W$ over~$k$ are called \emph{stably birationally equivalent}
if $V \times \A^r_k$ and $W \times \A^s_k$ are birationally equivalent
for some $r$, $s$.
It can be shown that for any finite group~$G$, the varieties~$\A^m_k/G$ and~$\SL_n/G$ appearing in Examples~\ref{ex:noetherisversal}
and~\ref{ex:noetherisversalsln},
for all values of~$m$, $n$ and all possible choices of embeddings $G \hookrightarrow S_m$ and $G \hookrightarrow \SL_n(k)$,
 all fall into the same stable birational equivalence class of varieties over~$k$.
This is the so-called ``no-name lemma'', see \cite[Corollary~3.9]{ct-sansuc-rationality}.
\end{rmk}

The notion of versality, in the context of these notes\footnote{Outside of the context discussed here,
the notion of versality notably also gives rise
to the definition of the ``essential dimension'' of a finite group~$G$ over a field~$k$---this is the minimal
dimension of a versal $G$\nobreakdash-torsor defined over~$k$---which is interesting in its own right
and has been the focus of many works (see \cite{buhlerreichstein,berhuyfavi,reichsteinicm,merkurjevsurveyed,merkurjevedbis,reichstein13}).
Even determining the essential dimension of $\Z/8\Z$ over~$\Q$
is a highly nontrivial task (see~\cite{florenceed}).}, is motivated by the following observation,
which is an improved version of
Corollary~\ref{cor:cortohilbert}:

\begin{prop}
\label{prop:versalrationalimpliesgrunwald}
Let~$k$ be a number field. Let $S_0 \subset \Omega$ be a finite subset.
Let~$G$ be a finite group.
Suppose that there exist an irreducible smooth quasi-projective variety~$Y$ over~$k$ and a free action of~$G$ on~$Y$ satisfying the following two conditions:
\begin{enumerate}[(i)]
\setlength\itemsep{.11em}
\item
the variety $X=Y/G$ satisfies weak approximation off~$S_0$, i.e.\ the diagonal embedding
$X(k) \hookrightarrow \prod_{v \in \Omega\setminus S_0}X(k_v)$ has dense image;
\item
the $G$\nobreakdash-torsor $\pi:Y\to X$ is weakly versal.
\end{enumerate}
Then Grunwald's problem admits a positive
answer for~$G$ over~$k$, for any finite subset $S \subset \Omega$ disjoint from~$S_0$.
\end{prop}

\begin{proof}
We shall need the following classical lemma, proved in
\cite[Proposition~3.5.74]{poonenqpoints}
and whose statement holds for any finite étale morphism~$\pi$.

\begin{lem}[Krasner]
\label{lem:krasner}
For  $v \in \Omega$,
the isomorphism class of the variety
$\pi^{-1}(x_v)$ over~$k_v$
is a locally constant function of $x_v \in X(k_v)$ with respect to the $v$\nobreakdash-adic
topology.
\end{lem}

Fix Galois field extensions $K_v/k_v$ for $v \in S$ as in Problem~\ref{pb:grunwald}
and
choose embeddings $\Gal(K_v/k_v) \hookrightarrow G$ for $v \in S$.
By Remark~\ref{rmk:pushtorsor},
these choices give rise to $G$\nobreakdash-torsors over~$k_v$ for $v \in S$.
By weak versality, the latter come from $k_v$\nobreakdash-points $x_v \in X(k_v)$.

Lemma~\ref{lem:krasner}
provides, for every $v \in S$, a neighbourhood $\sU_v \subset X(k_v)$
of~$x_v$
such that
 $\pi^{-1}(x_v')$
and $\pi^{-1}(x_v)$
are isomorphic,
as varieties over~$k_v$,
for all $x_v' \in \sU_v$.
In particular, by Remark~\ref{rmk:pushtorsor} again,
for all $v \in S$ and all $x_v'\in\sU_v$,
the fibre $\pi^{-1}(x_v')$ is isomorphic, over~$k_v$, to a disjoint union of copies of $\Spec(K_v)$.

The weak versality of~$\pi$ also implies that $X(k)\neq\emptyset$.
Theorem~\ref{th:ekedahl} therefore provides
a nonempty open subset
$\sU^0 \subset \prod_{v \in \Omega\setminus (S\cup S_0)} X(k_v)$ such that $\pi^{-1}(x)$ is irreducible
for all $x \in X(k)\cap \sU^0$.
Let $\sU=\big(\prod_{v \in S}\sU_v\big) \times \sU^0$.
As the variety~$X$ satisfies weak approximation off~$S_0$,
the set $X(k)\cap \sU$ is nonempty.
We fix $x \in X(k)\cap \sU$.

The fibre $\pi^{-1}(x)$ is now an irreducible $G$\nobreakdash-torsor
(i.e.\ $\Spec(K)$ for some Galois field extension~$K/k$ with Galois group~$G$)
whose scalar extension from~$k$ to~$k_v$, for each $v \in S$, is a disjoint union of copies of~$\Spec(K_v)$.
This proves the proposition.
\end{proof}

As smooth rational varieties satisfy weak approximation, one can apply
Proposition~\ref{prop:versalrationalimpliesgrunwald}
with $S_0=\emptyset$ whenever the variety $Y/G$ is rational and the torsor $Y \to Y/G$ is weakly versal.
In view of Example~\ref{ex:noetherisversal}, we deduce:

\begin{cor}
\label{cor:noetherimpliesgrunwald}
Given a finite group~$G$ and a number field~$k$,
a positive answer to Noether's problem for~$G$ and~$k$ implies a positive answer to Grunwald's problem for~$G$
and~$k$, for any $S \subset \Omega$.
\end{cor}

Corollary~\ref{cor:noetherimpliesgrunwald} was first established by Saltman (see \cite[Theorem~5.1, Theorem~5.9]{saltmangeneric}).
As an example of an application, Corollary~\ref{cor:noetherimpliesgrunwald} implies that over any number field~$k$,
Grunwald's problem has a positive answer for~$S_n$ and for~$A_5$ over~$k$,
without the need to exclude any place from $S \subset \Omega$
(see
Example~\ref{ex:noethersn} and Example~\ref{ex:noethera5}).

\subsection{Noether's problem: some counterexamples}
\label{subsec:noether:counter}

The hope that a positive solution to the inverse Galois problem might in general come
from a positive solution to Noether's problem turned out, however, to be too optimistic.
Indeed, Noether's problem seems to have a negative solution more often than not,
as we briefly discuss below.

\subsubsection{Counterexamples among abelian groups}

Noether's problem has a negative answer even for cyclic groups over~$\Q$.
Swan and Voskresenski\u{\i}
discovered, at the end of the 1960's, the counterexample $\Z/47\Z$ over~$\Q$
(see \cite{swan47}, \cite{voskbirlin}). An even smaller counterexample, the group $\Z/8\Z$ over~$\Q$,
was then exhibited
in~\cite{endomiyata}, \cite{lenstrainvent}, \cite{voskfieldofinvariants}.
As Saltman~\cite{saltmangeneric} later observed,
Corollary~\ref{cor:noetherimpliesgrunwald}
provides a direct proof that Noether's problem
admits a negative answer for~$\Z/8\Z$ over~$\Q$.
Indeed,
by this corollary,
it suffices to show that Grunwald's problem has a negative answer for $G=\Z/8\Z$, $k=\Q$ and $S=\emptyset$,
and this is exactly what Wang had done in the 1940's:

\begin{prop}[Wang]
\label{prop:wang}
In a cyclic field extension $K/\Q$
of degree~$8$,
the prime~$2$ cannot be inert.
In other words, the completion of a cyclic field extension $K/\Q$ of degree~$8$
at a place dividing~$2$ cannot
be the unramified extension of~$\Q_2$ of degree~$8$.
\end{prop}

An elementary proof can be found in \cite[p.~29, end of~\textsection5]{swansurvey}.

Further work on Noether's problem for abelian groups, by Endo, Miyata,
Voskresenski\u{\i} and Lenstra, led to a complete characterisation, by Lenstra~\cite{lenstrainvent}, of 
the stable rationality of the quotient $\A^n_k/G$ appearing in Problem~\ref{pb:noether} (and even
of its rationality, in the case where~$G$ acts through its regular representation), when~$G$ is a finite
abelian group
and~$k$ is an arbitrary field.   This characterisation is in terms of
the arithmetic of cyclotomic number fields.
For cyclic groups over~$\Q$, it reads as follows (see \cite[\textsection3]{lenstracyclic}):

\begin{thm}[Lenstra]
\label{th:lenstra}
Let~$n \geq 1$ be an integer.
Let~$G=\Z/n\Z$ faithfully act on~$\A^n_\Q$ by cyclically
permuting the coordinates.
The following conditions are equivalent:
\begin{enumerate}
\item The variety $\A^n_\Q/G$ is rational.
\item The variety $\A^n_\Q/G$ is stably rational.
\item
The integer~$n$ is not divisible by~$8$, and for every prime factor~$p$ of~$n$,
if $s$ denotes the $p$\nobreakdash-adic valuation of~$n$,
the cyclotomic ring $\Z\big[\zeta_{(p-1)p^{s-1}}\big]$ contains an element whose norm is equal to~$p$ or to~$-p$.
\end{enumerate}
\end{thm}

We recall that a variety is said to be \emph{stably rational} if its product with an affine space of large enough
dimension is rational.
Stable rationality is known
to be strictly weaker than rationality in general, even over~$\C$, see \cite{bctssd}.

Theorem~\ref{th:lenstra} when~$n$ is a prime number is due to
Voskresenski\u{\i}~\cite{voskprime}.  Even when~$n$ is prime, determining whether condition~(3)
of Theorem~\ref{th:lenstra} does or does not hold for a given~$n$ is in general a hard problem;
for instance, it is only recently that this condition was shown to fail for $n=59$ (see \cite[Added remark~3.2]{hoshicomputer}).
Even more recently, based on Theorem~\ref{th:lenstra}, on a height estimate due to Amoroso and Dvornicich~\cite{amorosodvornicich}
and on
extensive computer calculations run by Hoshi~\cite{hoshicomputer}, among other tools,
Plans~\cite{plans} was able to give a complete answer to Noether's problem for cyclic groups over~$\Q$:

\begin{thm}[Plans]
\label{th:plans}
The conditions of Theorem~\ref{th:lenstra} are equivalent to the following:
\begin{enumerate}
\setcounter{enumi}{3}
\item The integer~$n$ divides
$2^2 \cdot 3^m \cdot 5^2 \cdot 7^2 \cdot 11 \cdot 13 \cdot 17 \cdot 19 \cdot 23 \cdot 29 \cdot 31 \cdot 37 \cdot 41 \cdot 43 \cdot 61 \cdot 67 \cdot 71$
for some integer $m \geq 0$.
\end{enumerate}
In particular, Noether's problem has a negative answer over~$\Q$ for $G=\Z/p\Z$ for all but finitely many
prime numbers~$p$.
\end{thm}

\subsubsection{Counterexamples over~$\C$}
\label{subsubsec:counterexamplesoverc}

For non-abelian groups, Noether's problem has a negative answer even over~$\C$.  Saltman~\cite{saltmannoether}
gave the first counterexamples over~$\C$.
His results were then generalised by Bogomolov~\cite{bogomolovbrnr},
who established the following theorem
(see \cite[\textsection7]{ct-sansuc-rationality}
and \cite[\textsection6.6]{gilleszamuelycsa} for accounts of its proof):

\begin{thm}[Bogomolov's formula]
\label{th:bogomolovformula}
Let~$n \geq 1$ and $G \subset \SL_n(\C)$ be a finite subgroup.
The unramified Brauer group of the complex variety $\SL_n/G$ is isomorphic to
\begin{align}
\label{eq:bogomolovkernel}
\Ker\Big(H^2(G,\Q/\Z) \to \prod H^2(H,\Q/\Z)\Big)\rlap,
\end{align}
where the product ranges over all \emph{bicyclic} subgroups $H\subseteq G$ (i.e.\ abelian subgroups of~$G$
that are generated by at most two elements).
\end{thm}

We recall that the Brauer group, defined by Grothendieck as $H^2_\et(-,\Gm)$, is a stable birational invariant among smooth proper
varieties over a field of characteristic~$0$, and
 that the \emph{unramified Brauer group} of a variety over a field of characteristic~$0$
is by definition
the Brauer group of any smooth proper variety
birationally equivalent to it; for instance, the unramified Brauer group of~$\A^n_\C$ is trivial.
Thus, if the unramified Brauer group of a variety over~$\C$ does not vanish,
then this variety is not stably rational, a fortiori it is not rational.
The unramified Brauer group was first considered and used as a
 tool for rationality questions
by Saltman~\cite{saltmangenericmatrices,saltmannoether}.
For smooth proper unirational varieties over~$\C$, it coincides with
the invariant that had earlier been employed
by Artin and Mumford~\cite{artinmumford}
to give ``elementary'' examples of complex unirational threefolds failing to be
rational.
For a thorough treatment of the Brauer group, we refer the reader to \cite{ctskobook}.

In view of Remark~\ref{rmk:noname}, Bogomolov's formula gives an easy recipe for computing the
unramified
Brauer group of the variety $\A^n_\C/G$ that appears in Noether's problem over~$\C$.
The kernel~\eqref{eq:bogomolovkernel} can be computed to be nonzero
for some $p$\nobreakdash-groups~$G$, thus yielding counterexamples to Noether's problem over~$\C$
(see \cite[Example~7.5]{ct-sansuc-rationality}, \cite[\textsection6.7]{gilleszamuelycsa}).

Other counterexamples over~$\C$ were later produced by Peyre~\cite{peyrenoether} based on a further
stable birational invariant 
 introduced by
Colliot-Thélène and Ojanguren~\cite{ctojanguren}, called
 \emph{unramified cohomology of degree~$3$}.  The unramified Brauer
group coincides with unramified cohomology of degree~$2$.

Many more results about Noether's problem can be found in the survey~\cite{hoshisurvey}.

\subsection{Retract rationality}
Saltman introduced a useful weakening of the notion of stable rationality:
a variety~$X$ over a field~$k$ is said to be \emph{retract rational} 
if there exist an integer $n\geq 1$, a dense open subset $U \subseteq \A^n_k$
and a morphism $U \to X$ that admits a rational section.
Retract rationality is a stable birational invariant.

In the situation of Noether's problem, it can happen that the variety $\A^n_k/G$ fails to be rational
and even to be stably rational, but is nevertheless retract rational.
For instance this is so when $G=\Z/47\Z$ and $k=\Q$:

\begin{thm}[Saltman~\cite{saltmangeneric}]
\label{thm:saltmangeneric}
Taking up the notation of Problem~\ref{pb:noether},
assume that~$G$ is abelian, that~$k$ has characteristic~$0$,
and,
letting~$2^r$ denote the
highest power of~$2$ that divides the exponent of~$G$, that the cyclotomic field extension $k(\zeta_{2^r})/k$
is cyclic. Then the quotient $\A^n_k/G$ is retract
rational over~$k$.
\end{thm}

Theorem~\ref{thm:saltmangeneric} can in particular be applied to all finite abelian groups of odd order.
Thus, retract rationality is weaker than rationality
(compare with Theorem~\ref{th:plans}).
Nevertheless, as far as the applications to the inverse
Galois problem are concerned, it is just as good:
indeed,
as Saltman observed,
smooth retract rational varieties
over number fields are easily seen to satisfy weak approximation, so that
Proposition~\ref{prop:versalrationalimpliesgrunwald} can be applied with $S_0=\emptyset$ whenever the variety $Y/G$
is retract rational and the torsor $Y\to Y/G$ is weakly versal.

Combining this observation with Theorem~\ref{thm:saltmangeneric}
and with
Proposition~\ref{prop:versalrationalimpliesgrunwald},
we deduce, in view of Example~\ref{ex:noetherisversal},
that
 Grunwald's problem has a positive answer over any number field~$k$, without excluding any place,
for all abelian groups~$G$ satisfying the assumption of Theorem~\ref{thm:saltmangeneric}---a conclusion
that already resulted from the Grunwald--Wang theorem, but whose proof now fits into the framework
of Hilbert's and Noether's general strategy, even though, according to Theorem~\ref{th:plans},
 Noether's problem itself has a negative answer
for many of these groups~$G$ (perhaps even for ``almost all'' of them?).

Conversely, by the same token, Wang's negative answer to Grunwald's problem 
(see Proposition~\ref{prop:wang}) implies that
when $G=\Z/8\Z$ and $k=\Q$, the quotient $\A^n_k/G$ fails not only to be stably rational
but also to be retract rational.  Similarly, the negative answers to
Noether's problem over~$\C$ discussed in~\textsection\ref{subsubsec:counterexamplesoverc}
are in fact  counterexamples to the retract rationality of the quotients $\A^n_\C/G$ in question.
Thus, despite the wider scope of applicability of the Hilbert--Noether method when rationality
is replaced with the weaker notion of retract rationality, further ideas are necessary to address arbitrary finite groups.

\section{Regular inverse Galois problem}
\label{sec:rig}

\subsection{Statement}
\label{subsec:rigstatement}

We saw, in~\textsection\ref{sec:1},
that Noether's problem does not always admit a positive answer,
i.e.\ the quotient variety $\A^n_k/G$ can fail to be rational, or stably rational, or even
retract rational.  A simple way out, if one still wants to apply Hilbert's irreducibility theorem,
is to look
for rational \emph{subvarieties} of~$\A^n_k/G$, in particular rational curves.  To take advantage of the
geometry of the situation, it is natural to focus on those rational
curves whose inverse image in~$\A^n_k$ is geometrically irreducible
and meets the locus $Y \subset \A^n_k$ on which~$G$ acts freely.  By the versality of the torsor $Y \to Y/G$
(Example~\ref{ex:noetherisversal}),
finding such curves
is the same as solving the \emph{regular inverse Galois problem} (when~$k$ is perfect):

\begin{problem}[regular inverse Galois]
\label{pb:rig}
Let~$k$ be a field.
Let~$G$ be a finite group.
Do there exist a smooth, projective, geometrically irreducible curve~$C$ over~$k$
and a finite morphism $\pi:C\to \P^1_k$ such that the corresponding extension of function fields $k(C)/k(t)$
is Galois with $\Gal(k(C)/k(t))\simeq G$?
\end{problem}

When~$k$ is a perfect field, this is equivalent to asking for the existence of
a field extension of~$k(t)$ with Galois group~$G$ 
in which~$k$ is algebraically closed,
i.e.\ a field extension that is regular over~$k$.  Following standard practice, we shall refer to such a field extension
as a \emph{regular Galois extension of~$k(t)$ with group~$G$}.

When~$k$ is a number field,
a positive answer to Problem~\ref{pb:rig} for~$k$ and~$G$ implies a positive answer to the inverse Galois
problem for~$k$ and~$G$,
by Hilbert's irreducibility theorem.
Over an arbitrary field and for an arbitrary finite group,
the inverse Galois problem and Noether's problem  both have negative answers in general, as we have seen;
in contrast,
Problem~\ref{pb:rig} might well always have a positive answer.

\begin{rmk}
\label{rmk:noetherimpliesrig}
It follows from the Bertini theorem that if~$k$ is infinite and perfect,
a positive answer to Noether's problem for~$k$ and~$G$ implies a positive answer to the regular inverse Galois
problem for~$k$ and~$G$
(see \cite[Théorème~6.3]{jouanoloubertini}).
In fact, for such~$k$, one can check that the retract rationality of $\A^n_k/G$ already
implies a positive answer to the regular inverse Galois
problem for~$k$ and~$G$.
\end{rmk}

\subsection{Riemann's existence theorem}

A solution to the regular inverse Galois problem over~$k$ gives rise, by scalar extension, to a solution
over any field extension of~$k$.  Thus, in order to find a solution over~$\Q$, it is necessary to first
solve the problem over~$\C$ and over~$\bar \Q$.
The key tool for this is Riemann's existence theorem, which allows one
to transform this algebraic question into a purely topological one.

\begin{thm}[Riemann's existence theorem]
\label{th:ret}
Let~$k$ be an algebraically closed subfield of~$\C$.
Let~$X$ be a variety over~$k$.  The natural functor
\begin{align*}
\Big(\text{étale coverings of~$X$}\Big)
\to \Big(\text{finite topological coverings of $X(\C)$}\Big)
\end{align*}
that maps $Y \to X$ to $Y(\C) \to X(\C)$ is an equivalence of categories.
\end{thm}

An \emph{étale covering} of~$X$ is a variety over~$k$ endowed with a finite étale morphism to~$X$.
A~topological covering is \emph{finite} if its fibres are finite.  Theorem~\ref{th:ret} in the above
formulation
is proved in \cite{sga1}.
To be precise, the case where $k=\C$
is \cite[Exp.~XII, Théorème~5.1]{sga1}
and builds on Grothendieck's reworking of Serre's GAGA
theorems; the case of an arbitrary algebraically closed subfield of~$\C$
then results from it by \cite[Exp.~XIII, Corollaire~3.5]{sga1}.

\begin{cor}
\label{cor:ret}
Let~$k$ be an algebraically closed subfield of~$\C$.
Let~$X$ be a connected variety over~$k$.
Let $x \in X(k)$.  For any finite group~$G$,
isomorphism classes of $G$\nobreakdash-torsors (resp.\ of connected $G$\nobreakdash-torsors)
$Y\to X$
endowed with a lift $y\in Y(k)$ of~$x$
are canonically in one-to-one correspondence with homomorphisms $\pi_1(X(\C),x) \to G$
(resp.\ with surjective homomorphisms
 $\pi_1(X(\C),x) \twoheadrightarrow G$).
Changing the choice of~$y$ amounts to conjugating the homomorphism
by an element of~$G$.
\end{cor}

\begin{proof}
Indeed, this
 follows from Theorem~\ref{th:ret} combined with the well-known
equivalence of categories between the category of  topological
coverings of~$X(\C)$ and the category of sets endowed with an action of~$\pi_1(X(\C),x)$
(see \cite[Theorem~2.3.4]{szamuelygaloisgroups}).
The homomorphism $\pi_1(X(\C),x) \to G$ corresponding to $Y \to X$
sends $\gamma \in \pi_1(X(\C),x)$ to the unique $g \in G$ such that $\gamma y = yg$,
where we are taking the convention that the action of~$G$ on~$Y$ is a right action
and that the monodromy action of $\pi_1(X(\C),x)$ on the fibre of $Y(\C)\to X(\C)$ above~$x$
is a left action.
\end{proof}

\begin{rmk}[reminder on monodromy groups and Galois groups]
\label{rk:ret}
Let~$k$ and~$X$ be as in Corollary~\ref{cor:ret}.
Let $x \in X(\C)$.
The \emph{monodromy group}~$M$ of
an étale covering $Y\to X$ is, by definition,
 the largest quotient of~$\pi_1(X(\C),x)$ through which the monodromy action of this group on
the fibre of $Y(\C) \to X(\C)$ above~$x$ factors.
Assume that~$X$ is normal and irreducible
and let $Y' \to Y \to X$
be a tower of irreducible étale coverings
such that the field extension $k(Y')/k(X)$ is a Galois
closure of $k(Y)/k(X)$.
Let $G=\Gal(k(Y')/k(X))$.
Then $Y' \to X$ is the normalisation of~$X$ in~$k(Y')$; as such, it receives an action of~$G$,
with respect to which it is a $G$\nobreakdash-torsor;
in addition, the surjective homomorphism $\pi_1(X(\C),x)\twoheadrightarrow G$
that Corollary~\ref{cor:ret}
associates with $Y' \to X$ and with the choice of a lift $y' \in Y'(k)$ of~$x$
induces an isomorphism $M \isoto G$.
(Changing the choice of the lift~$y'$ amounts to composing this isomorphism with an inner automorphism.)
Thus, computing the Galois group of the Galois closure of the field extension $k(Y)/k(X)$
is tantamount to computing a monodromy group in the topological setting.
\end{rmk}

\subsection{Classifying Galois covers of the projective line over~\texorpdfstring{$\C$}{𝐂} or over~\texorpdfstring{$\bar \Q$}{Q̄}}
\label{subsec:classifyc}

Let us apply Theorem~\ref{th:ret} to the open subsets of the projective line.
The fundamental group
of the complement of finitely many points in~$\P^1(\C)$ is easy to describe:

\begin{prop}
\label{prop:descpi1}
Let $X \subseteq \P^1_\C$ be a dense open subset.
Write $\P^1_\C \setminus X = \{b_1,\dots,b_r\}$.
Let $x \in X(\C)$.  The group $\pi_1(X(\C),x)$ admits a presentation with~$r$
generators $\gamma_1,\dots,\gamma_r$ and a unique relation $\gamma_1 \cdots \gamma_r=1$,
such that~$\gamma_i$ belongs, for every $i\in\{1,\dots,r\}$,
to the conjugacy class in~$\pi_1(X(\C),x)$
of a local counterclockwise loop around~$b_i$.
\end{prop}

What the last sentence of Proposition~\ref{prop:descpi1} means is this:
if~$N_i$ denotes
a small enough open neighbourhood of~$b_i$ in~$\P^1(\C)$ that is biholomorphic to the unit disc, then
a loop contained in $N_i \setminus \{b_i\}$ and going once around~$b_i$ in the counterclockwise direction
gives rise, after choosing a path from~$x$ to a point of this loop, to an element of $\pi_1(X(\C),x)$
whose conjugacy class does not depend on the chosen path.
The content of Proposition~\ref{prop:descpi1} is that these paths can be chosen in such a way
that the~$\gamma_i$ generate~$\pi_1(X(\C),x)$ and satisfy the relation $\gamma_1\cdots\gamma_r=1$.
This is elementary and well-known.

Using Proposition~\ref{prop:descpi1}, we can draw the following corollary
from Riemann's existence theorem.
Corollary~\ref{cor:retgcovers}
completely describes $G$\nobreakdash-torsors over dense open subsets
of the projective line over algebraically closed subfields of~$\C$,
and implies a positive solution
to the regular inverse Galois problem over such fields.
(The notation $\ni^\ast_r(G)$ appearing in its statement
refers to the name Nielsen,
see \cite[\textsection9.2]{volklein}, \cite[\textsection3.1]{romagnywewers}.)

\begin{cor}
\label{cor:retgcovers}
Let $k$ be an algebraically closed subfield of~$\C$.
Let $X \subseteq \P^1_k$ be a dense open subset.
Write $\P^1_\C \setminus X = \{b_1,\dots,b_r\}$.
Let~$G$ be a finite group.
Consider the set
of $r$\nobreakdash-tuples $(g_1,\dots,g_r) \in G^r$ such that
$g_1\cdots g_r=1$ and that $g_1,\dots,g_r$ generate~$G$.
Let $\ni^\ast_r(G)$ denote the quotient of this set
by the action of~$G$ by simultaneous conjugation.
The set of isomorphism classes of irreducible $G$\nobreakdash-torsors
over~$X$ is in bijection with~$\ni^\ast_r(G)$ (through a bijection that is canonically determined
once a presentation of~$\pi_1(X(\C),x)$
as in Proposition~\ref{prop:descpi1} is fixed).
\end{cor}

\begin{proof}
By Corollary~\ref{cor:ret}, isomorphism classes of irreducible $G$\nobreakdash-torsors
over~$X$ are canonically in one-to-one correspondence
with conjugacy classes of surjections $\pi_1(X(\C),x)\twoheadrightarrow G$.
Apply Proposition~\ref{prop:descpi1} to conclude.
\end{proof}

\begin{cor}
For any finite group~$G$,
the regular inverse Galois problem admits a positive answer over~$\bar \Q$.
\end{cor}

\begin{proof}
Let~$r$ be an integer, large enough that~$G$ can be generated by~$r-1$ elements.
Pick~$r$ points of~$\P^1(\bar \Q)$
and let $X \subset  \P^1_{\bar \Q}$ denote their complement.
As $\ni^\ast_{r}(G)\neq\emptyset$,
 Corollary~\ref{cor:retgcovers}
ensures the existence of an irreducible $G$\nobreakdash-torsor $p:Y \to X$.
As~$Y$ is normal and~$p$  is finite,
the normalisation of~$\P^1_{\bar \Q}$ in the function field of~$Y$
is a smooth curve~$C$ over~$\bar \Q$ containing~$Y$
as a dense open subset, equipped with a finite morphism $\pi:C \to \P^1_{\bar \Q}$
that extends~$p$.
As~$p$ is a $G$\nobreakdash-torsor, the function field extension $\bar \Q(C)/\bar \Q(t)$
is Galois with group~$G$
(see~\textsection\ref{subsec:torsorsandgalois}).
\end{proof}

\subsection{Monodromy of some non-Galois covers of the projective line}

Proposition~\ref{prop:descpi1} is also useful for computing
the monodromy of ramified covers of the complex projective line that are not necessarily Galois,
via the following result.

\begin{prop}
\label{prop:ssgsd}
Let~$C$ be a smooth, projective, irreducible curve over~$\C$,
endowed with a finite morphism $\pi:C \to \P^1_{\C}$.
Let $X \subseteq \P^1_\C$
be a dense open subset over which~$\pi$ is étale.
Fix $x \in X(\C)$ and write $\P^1_\C \setminus X = \{b_1,\dots,b_r\}$.
Let~$M$ denote the monodromy group of~$\pi$, i.e.\ the largest quotient of $\pi_1(X(\C),x)$
that still acts on~$\pi^{-1}(x)$.
After choosing a bijection $\pi^{-1}(x) \simeq \{1,\dots,n\}$, we view~$M$ as a transitive subgroup of
the symmetric group~$S_n$.
There exist
 $\mu_1,\dots,\mu_r \in M$  satisfying the following three properties:
\begin{enumerate}
\item the elements $\mu_1,\dots,\mu_r$ generate the group~$M$;
\item their product $\mu_1\cdots\mu_r$ is the identity of~$M$;
\item for each $i \in \{1,\dots,r\}$, the element $\mu_i \in S_n$ is a product
of cycles whose lengths are the ramification indices of~$\pi$ at the points of $\pi^{-1}(b_i)$.
\end{enumerate}
\end{prop}

\begin{proof}
Applying
Proposition~\ref{prop:descpi1}
and letting~$\mu_i$ denote the image of~$\gamma_i$ in~$M$, we obtain~(1) and~(2).
Property~(3) only depends on the conjugacy class of~$\gamma_i$
and is a standard calculation of the monodromy of the étale coverings of the punctured
unit disc.
\end{proof}

\begin{example}
\label{ex:rigsn}
Let~$C$ be a smooth, projective, irreducible curve over an algebraically closed subfield~$k$ of~$\C$, endowed with a morphism
$\pi:C \to \P^1_k$ of degree $n \geq 1$.
Assume that all ramification points have ramification index~$2$ and that
no two of them lie in the same fibre of~$\pi$.
Then the Galois group of a Galois closure of the function field extension
$k(C)/k(t)$ is the full symmetric group~$S_n$.  Indeed,  Remark~\ref{rk:ret}
and Proposition~\ref{prop:ssgsd} show that this Galois group is a transitive subgroup of~$S_n$
generated by transpositions; the only such subgroup is~$S_n$ itself.
\end{example}

\begin{rmk}
\label{rk:geomconnected}
Let~$k$ be a field of characteristic~$0$.
Let $k(t) \subseteq K \subseteq K' \subset \overline{k(t)}$ be a tower of fields,
where~$\overline{k(t)}$ denotes an algebraic closure of~$k(t)$,
where~$K/k(t)$ is a finite extension and where~$K'/k(t)$ is its Galois closure
inside~$\overline{k(t)}$.
Let us assume that~$k$ is algebraically closed in~$K$. The field~$k$ need not, in general, be algebraically closed in~$K'$.
(For example, if $k=\Q$ and $K=\Q(t^{1/n})$, then $K'=\Q(\zeta_n)(t^{1/n})$, where~$\zeta_n$ denotes a primitive
$n$th root of unity.)
This pathology, however, cannot occur if the underlying topological monodromy
 group is the full symmetric group, or, more generally, if it is a self-normalising subgroup of the ambient
symmetric group.
Indeed, let~$k'$ denote the algebraic closure of~$k$ in~$K'$, set $G=\Gal(K'/k(t))$ and $G_{\geom}=\Gal(K'/k'(t))$.
Letting~$\bar k$ denote the algebraic closure of~$k$ in~$\overline{k(t)}$, we
remark that $K' \otimes_{k'} \bar k$ and $K \otimes_k \bar k$ are fields and that the field extension
 $K' \otimes_{k'} \bar k / \bar k(t)$
is a Galois closure of
 $K \otimes_k \bar k / \bar k(t)$, so that its Galois group $G_{\geom}$
can be viewed as the topological monodromy group
associated with $K/k(t)$
(see Remark~\ref{rk:ret}).
Fix a primitive element $\alpha_1 \in K$ over~$k(t)$.
Denote by $\alpha_1,\dots,\alpha_n \in K'$ the collection of its Galois conjugates.
As~$G$ acts faithfully on the~$\alpha_i$'s, there is a sequence of inclusions
$G_{\geom} \subseteq G \subseteq S_n$.   As~$k'/k$ is a Galois field extension,
the group~$G_{\geom}$ is normal in~$G$; hence,
if~$G_{\geom}$ is self-normalising in~$S_n$,
then
 $G=G_{\geom}$ 
and~$k$ is algebraically closed in~$K'$.  Thus, for instance,
if the curve~$C$ and the morphism~$\pi$ of Example~\ref{ex:rigsn} come by scalar extension
from a curve and a morphism defined over~$\Q$, and if~$K/\Q(t)$ denotes the function field extension
given by the latter morphism, then a Galois closure of~$K/\Q(t)$ has Galois group~$S_n$.
\end{rmk}

In conjunction with Remark~\ref{rk:geomconnected},
Example~\ref{ex:rigsn} leads to many concrete examples of regular Galois extensions of~$\Q(t)$ with group~$S_n$.
Let us recall, however, that the mere existence of
 regular Galois extensions of~$\Q(t)$ with group~$S_n$
already followed
from the positive answer to Noether's problem for~$S_n$
over~$\Q$ (see Example~\ref{ex:noethersn} and
 Remark~\ref{rmk:noetherimpliesrig}).
As Noether's problem is open for the alternating group~$A_n$ over~$\Q$ as soon as $n \geq 6$,
it is of interest to note that Proposition~\ref{prop:ssgsd} also leads to concrete examples of regular
Galois extensions of~$\Q(t)$ with group~$A_n$ for all values of~$n$, as we now illustrate.

\begin{example}
\label{ex:alternating}
Let~$C$ be a smooth, projective, geometrically irreducible curve over a subfield~$k$ of~$\C$, endowed with a morphism
$\pi:C \to \P^1_k$ of degree $n \geq 3$.
Assume that~$\pi$ has exactly three ramification points,
that these ramification points are rational points of~$C$ lying above $0,1,\infty \in \P^1(k)$, with ramification indices~$e_0$, $e_1$, $e_\infty$, respectively,
and that $(e_0,e_1,e_\infty)=(n,n-1,2)$ if~$n$ is even
and $(e_0,e_1,e_\infty)=(n-1,n,2)$ if~$n$ is odd.
Let~$K'/k(t)$ denote a Galois closure of
 the function field extension
$k(C)/k(t)$.
We first note that~$K'$ is a regular Galois extension of~$k(t)$ with group~$S_n$.
Indeed, when~$k$ is algebraically closed, Remark~\ref{rk:ret}
and Proposition~\ref{prop:ssgsd} imply that $\Gal(K'/k(t))$ is a transitive subgroup of~$S_n$
that contains a cycle of order~$n-1$ and a transposition, but the only such subgroup is~$S_n$ itself (see
\cite[Lemma~4.4.3]{serretopics}); by Remark~\ref{rk:geomconnected}, the case of arbitrary~$k$ follows.
Secondly, we claim that there exists $\alpha \in k^*$
such that $\alpha t$ is a square in~$K'$.
Setting $u=\sqrt{\alpha t}$, this will imply that~$K'$ is a regular Galois extension of~$k(u)$
with group~$A_n$ (where~$u$ can now be viewed as a free variable), as desired.
To verify this claim, we note that the topological monodromy of the double cover of~$\P^1_k$
corresponding to the (unique) quadratic subextension~$L/k(t)$ of~$K'/k(t)$
is obtained by composing the topological monodromy of~$\pi$ with the signature morphism $S_n \to \Z/2\Z$.
As the local monodromy of~$\pi$ at~$1$ is given by a cycle of odd length,
it follows that~$L/k(t)$
is unramified outside of~$0$ and~$\infty$, and, hence, that $L=k\big(\sqrt{\alpha t}\mkern1mu\big)$ for some $\alpha \in k^*$
(as~$k$ is algebraically closed in~$L$).
\end{example}

For explicit equations to which Example~\ref{ex:alternating} can be applied,
see \cite[\textsection4.5]{serretopics}.

\subsection{Looking for covers over non-algebraically closed ground fields}

Now that we know that the regular inverse Galois problem has a positive answer
over~$\bar \Q$, we can try to
find solutions over~$\Q$ or at least over overfields of~$\Q$ as small as possible.
This has been achieved over the completions of~$\Q$,
thus yielding, for all finite groups, a positive answer to the regular inverse Galois problem
over~$\R$
(Krull and Neukirch~\cite{krullneukirch}) and over the field~$\Q_p$ of $p$\nobreakdash-adic numbers
for every prime~$p$ (Harbater~\cite{harbaterinvgal}).
Pop~\cite{poplarge} generalised these results as follows\footnote{It is not immediately clear that
the article \cite{poplarge} 
 establishes
Theorem~\ref{th:pop} as we have stated it, without assuming the field to be perfect:
in our definition of the regular inverse
Galois problem,  the sought-for field extension of~$k(t)$
was required to admit a \emph{smooth} projective model, which could fail over imperfect fields.
However, in any case, Theorem~\ref{th:pop} as we have stated it is proved in \cite[Théorème~1.1]{moretbaillyconstruction}.}\textsuperscript{,}\footnote{Theorem~\ref{th:pop} (at least for perfect large fields, see the previous footnote) also follows from the results of Harbater~\cite{harbaterinvgal}, see
\cite[\textsection4.5]{harbatericm}.}:

\begin{thm}[Harbater and Pop]
\label{th:pop}
The regular inverse Galois problem has a positive answer over any large field, for any finite group.
\end{thm}

By definition, a field~$k$ is \emph{large} when every smooth curve over~$k$ that has a rational point
has infinitely many of them.  Examples include all fields that are complete with respect to an absolute value,
such as~$\R$ and~$\Qp$, as well as infinite algebraic extensions of finite fields or more generally
all so-called pseudo-algebraically closed fields (fields over which every smooth geometrically connected
curve has infinitely many rational points).

The proofs of Theorem~\ref{th:pop} given by Harbater and by Pop rely, in the formal or in
the rigid analytic context over a
 complete discretely valued ground
field,
on the construction, by patching, of appropriate ``topological coverings'',
and on a variant, in the corresponding context, of Riemann's existence theorem.  Over~$\C$,
the underlying patching construction is presented in \cite[\textsection3.5]{szamuelygaloisgroups}.

Theorem~\ref{th:pop} had previously been established by Fried and Völklein~\cite{friedvolkleinmoduli}
in the case of pseudo-algebraically closed fields of characteristic~$0$.  From this special case they
 deduced the following result in positive characteristic:

\begin{thm}[Fried and Völklein]
Let~$G$ be a finite group.  The regular inverse Galois problem has a positive answer for~$G$
over~$\F_p(t)$ for all but finitely many primes~$p$.
\end{thm}

Colliot-Thélène later
shed new light on Theorem~\ref{th:pop} by recasting it
as a theorem about the existence of suitable rational curves
on the varieties $\A^n_k/G$ appearing in Noether's problem, and by noting that
even though these varieties can fail to be rational,
they are in any case \emph{rationally connected}, which opens the door
to applications of the theory of deformation of rational
curves on rationally connected varieties over large fields---a
theory developed, in great generality, by Kollár~\cite{kollarloc}.
Over large fields of characteristic~$0$,
a geometric proof of Theorem~\ref{th:pop} that proceeds by constructing rational curves
on~$\A^n_k/G$ was thus given in~\cite{ctrc}.
See also
\cite{kollarfundamentalgroups}, \cite{kollarfundamentalgroups2},
\cite{moretbaillyconstruction} for generalisations.

Unfortunately,
no method for the systematic construction of rational curves on rationally
connected varieties over~$\Q$ is known;
more generally,
the various methods on which all known proofs of Theorem~\ref{th:pop} rely
fall short
of solving any case of the regular inverse Galois problem over a given number field.
As of today, all known constructions of realisations of finite groups as regular Galois groups over~$\Q$ exploit
more or less ad hoc ideas.  One of the most successful approaches is the rigidity method,
initiated by Shih, Fried, Belyi, Matzat and Thompson in the 1970's and the 1980's, which we discuss next,
in \textsection\textsection\ref{subsec:hurwitzspaces}--\ref{subsec:rigidity}.

\begin{rmk}
In practice, the rigidity method is particularly useful for realising finite simple groups.
Considering the regular inverse Galois problem
for various classes of non-simple groups, especially central extensions,
also leads to interesting challenges,
better dealt with by other methods.
We refer the reader to the work of Mestre~\cite{mestreantilde, mestresl2f7m12tilde, mestre6a66a7},
who solved it
for the central extensions of~$A_n$ for all~$n$, as well as for~$\SL_2(\F_7)$ and for the
unique nontrivial
central extension of the Mathieu group~$M_{12}$ by~$\Z/2\Z$
(starting from
known regular realisations of~$\PSL_2(\F_7)$ and of~$M_{12}$).
To put these results into perspective,
let us note that Noether's problem has negative answers, over~$\Q$, for the
unique nontrivial central extensions of~$A_6$ and of~$A_7$ by~$\Z/2\Z$ (Serre \cite[Theorem~33.25]{garibaldimerkurjevserre}).
\end{rmk}

\subsection{Hurwitz spaces}
\label{subsec:hurwitzspaces}

Even though Hurwitz spaces are not necessary for the description and the implementation of the rigidity method,
their introduction makes the theory rather transparent; in addition, they are indeed indispensable
for some of its refinements.
Hurwitz spaces are moduli spaces of smooth projective irreducible covers of the
projective line.  We shall consider them only in characteristic~$0$.
In addition, we shall restrict attention to the moduli space of $G$\nobreakdash-covers;
the term ``$G$\nobreakdash-cover'' is another
name for the regular Galois extensions of~$k(t)$ with group~$G$ that we have been considering
since the beginning of~\textsection\ref{sec:rig}:

\begin{defn}
\label{def:gcover}
Let~$G$ be a finite group.  Let~$k$ be a field.
A \emph{$G$\nobreakdash-cover} over~$k$
is a smooth, proper, geometrically irreducible curve~$C$ over~$k$
endowed, on the one hand, with a finite morphism $\pi:C\to \P^1_k$
such that the corresponding extension of function fields $k(C)/k(t)$
is Galois, and, on the other hand,
with an isomorphism $G \isoto \Gal(k(C)/k(t))$.
 (In particular~$G$ acts faithfully on~$C$
and the morphism $\pi^{-1}(U) \to U$ induced by~$\pi$
is a $G$\nobreakdash-torsor for any dense open subset
$U \subset \P^1_k$ above which~$\pi$ is étale.)
\end{defn}

The group of automorphisms of any $G$\nobreakdash-cover, i.e.\ the group of automorphisms of~$C$ that respect not only
the morphism~$\pi$
but also the given isomorphism $G \isoto \Gal(k(C)/k(t))$,
is the centre of~$G$.  We shall assume,  until the end of~\textsection\ref{sec:rig},
that~$G$ has trivial centre.  This is not too serious a restriction (as any
finite group is a quotient of a finite group with trivial centre,
see \cite[Lemma~2]{friedvolkleinmoduli})
and it will ensure that our
moduli space is a variety rather than a stack
(as the objects that we want to classify have no
nontrivial automorphism).

To prepare for the statement of the next theorem, we need to introduce some notation.
When~$k$ has characteristic~$0$,
the \emph{branch locus} of $\pi:C\to\P^1_k$ is by definition the smallest
reduced
$0$\nobreakdash-dimensional subvariety~$B$ of~$\P^1_k$ such that~$\pi$ is étale over $\P^1_k\setminus B$.
Its \emph{degree} is the cardinality of $B(\bar k)$, where~$\bar k$ denotes an algebraic closure of~$k$.
For any integer $r \geq 1$, we denote by
 $\sU^r \subset (\P^1_\Q)^r$ the locus of $r$\nobreakdash-tuples with  pairwise distinct components,
and by~$\sU_r$ the quotient of $\sU^r$ by the natural
action of the symmetric group~$S_r$.  Thus~$\sU_r$ is a smooth variety over~$\Q$,
and for any field~$k$ of characteristic~$0$,
the set $\sU_r(k)$ can be  identified with the set of 
reduced
$0$\nobreakdash-dimensional subvarieties of~$\P^1_k$ of degree~$r$,
i.e.\ with the set of subsets of $\P^1(\bar k)$ of cardinality~$r$ that are stable under
$\Gal(\bar k/k)$.

\begin{thm}[Fried and Völklein \cite{friedvolkleinmoduli}]
\label{th:hurwitzspaces}
Let~$G$ be a finite group with trivial centre
and $r\geq 1$ be an integer.
With~$G$ and~$r$,
one can canonically associate
a smooth variety~$\sH_{G,r}$ over~$\Q$ such that for any field~$k$ of characteristic~$0$,
the set $\sH_{G,r}(k)$ is the set of isomorphism classes of $G$\nobreakdash-covers over~$k$
whose branch locus has degree~$r$.
It is equipped with a finite étale morphism $\rho:\sH_{G,r} \to \sU_r$ that maps the isomorphism class
of a $G$\nobreakdash-cover to its branch locus.
\end{thm}

The variety $\sH_{G,r}$ is called a \emph{Hurwitz space}.

A modern approach to Theorem~\ref{th:hurwitzspaces} consists
in defining $G$\nobreakdash-covers
 not just over fields, as in Definition~\ref{def:gcover},
but more generally over schemes; one then proves
that the resulting moduli
functor on the category of schemes of characteristic~$0$
is representable, by~$\sH_{G,r}$.
This is the approach adopted by
Wewers~\cite{wewersphd}, who works more generally over~$\Z$ (with tame covers)
and without assuming that the centre
of~$G$ is trivial (thus obtaining a moduli stack $\sH_{G,r}$).  See \cite{romagnywewers}.

We note that
Hurwitz spaces were first contemplated by Hurwitz~\cite{hurwitzorig}, and, with a functorial point of view,
by Fulton~\cite{fultonhurwitz}. However, these authors only considered
(non-Galois) covers with ``simple'' ramification, i.e.\ such
that all ramification points have ramification index~$2$
and no two of them lie over the same branch point. This is insufficient for the purposes
of the regular inverse Galois problem (see Example~\ref{ex:rigsn}).

Let us come back to our motivation.
It is tautological that for any finite group~$G$ with trivial centre, the regular inverse Galois problem
admits a positive answer for~$G$ over~$\Q$ if and only if there exists an
integer $r \geq 1$ such that $\sH_{G,r}(\Q)\neq\emptyset$.
For the question of the existence of a rational point
on one of the varieties
 $\sH_{G,r}$
to be tractable, one needs, in turn, some understanding of their geometry.

The varieties~$\sH_{G,r}$ can be described in a very explicit combinatorial fashion, at least
geometrically,
thanks to the finite étale morphism
$\rho:\sH_{G,r} \to \sU_r$
given by Theorem~\ref{th:hurwitzspaces}.
Let us pick up the notation $\ni^\ast_r(G)$ introduced in Corollary~\ref{cor:retgcovers}
and write
 $\ni_r(G) \subseteq \ni^\ast_r(G)$ for the subset formed by the conjugacy classes of those $r$\nobreakdash-tuples
$(g_1,\dots,g_r)$ such that none of the~$g_i$'s is equal to~$1$.
It then follows from
(the proof of) Corollary~\ref{cor:retgcovers}
that
the fibre of~$\rho$
above any complex point of~$\sU_r$ can be identified with~$\ni_r(G)$.
In addition, the fundamental group of $\sU_r(\C)$ admits a down-to-earth presentation (as a quotient
of the Artin braid group by one relation) and its action on~$\ni_r(G)$ can also be made explicit
(see \cite[\textsection1.3]{friedvolkleinmoduli}).
Thus, for instance, the task of describing the irreducible components of the variety~$(\sH_{G,r})_{\bar \Q}$
becomes equivalent to that of computing the orbits of a certain action of the braid group on~$\ni_r(G)$.
Unfortunately, as~$r$ increases, even this ``simple'' task quickly becomes computationally infeasible for modern computers (see e.g.\ \cite{haefner}).

\subsection{The rigidity method}
\label{subsec:rigidity}

This method consists in cleverly
identifying irreducible components of~$\sH_{G,r}$ that contain rational points
for somehow ``trivial'' reasons.
To explain it, we need to refine the étale covering~$\rho$ that appears in Theorem~\ref{th:hurwitzspaces}.

\subsubsection{Algebraic local monodromy}
\label{subsubsec:alglocalmon}

Let~$k$ be a field of characteristic~$0$.
Let $\pi:C \to \P^1_k$ be a $G$\nobreakdash-cover over~$k$.
Let $b_i \in \P^1(k)$ be a rational branch point.
Let $X \subset \P^1_k$ be a dense open subset over which~$\pi$ is étale.

Under the assumption that~$k$ is a subfield of~$\C$,
we have associated with~$\pi$ and~$b_i$,
in~\textsection\ref{subsec:classifyc},
a canonical conjugacy class of~$G$,
namely the conjugacy class of the element~$g_i$ appearing
in Corollary~\ref{cor:retgcovers}.
This is the ``local monodromy'' of~$\pi$ at~$b_i$.
We recall that it  is the image,
by a surjection $\pi_1(X(\C),x)\twoheadrightarrow G$ that is well-defined up to conjugation,
of the conjugacy class of a local counterclockwise loop around~$b_i$.
To make this topological definition fit in with the moduli picture of Theorem~\ref{th:hurwitzspaces}
and in particular to understand how it
behaves
with respect to the action of the group of automorphisms of~$k$,
we need to make it algebraic.

We do this as follows.
Let~$\bar k$ be an algebraic
closure of~$k$.
  The completion $\bar k(t)_{b_i}$
of~$\bar k(t)$ at the discrete valuation defined by~$b_i$ is isomorphic
to the field of formal power series~$\bar k((u))$, whose algebraic
closure is the field of Puiseux series $\bigcup_{n\geq 1} \bar k((u^{1/n}))$.
By Kummer theory,
the absolute Galois group of $\bar k(t)_{b_i}$ is
canonically isomorphic to $\Zhat(1)_{\bar k}=\varprojlim_{n\geq 1} \mmu_n(\bar k)$. The inclusion of fields $\bar k(t) \hookrightarrow \bar k(t)_{b_i}$ induces a continuous
 homomorphism in the reverse direction,
well-defined up to conjugation,
 between their absolute Galois groups;
hence it induces a continuous
homomorphism $\hat\Z(1)_{\bar k} \to G$,
well-defined up to conjugation by an element of~$G$.
This conjugacy class of homomorphisms $\hat\Z(1)_{\bar k} \to G$ is the analogue of the~$g_i$ from Corollary~\ref{cor:retgcovers}.  We call it the \emph{algebraic local monodromy} of~$\pi$ at~$b_i$.

\begin{rmks}
\label{rmks:zhat1zhat}
(i)
When~$k$ is a subfield of~$\C$, one can use the generator $\zeta_n=e^{2i\pi/n}$
of~$\mmu_n(\bar k)$ to identify~$\Zhat(1)_{\bar k}$ with~$\Zhat$ as topological groups (i.e.\ disregarding Galois actions).
The algebraic local monodromy of~$\pi$ at~$b_i$
then becomes identified with 
 a conjugacy class of~$G$.
One verifies that through this identification, the algebraic local monodromy of~$\pi$ at~$b_i$
coincides with the conjugacy class of~$g_i$ from Corollary~\ref{cor:retgcovers}.

(ii)
The natural action of $\Gal(\bar k/k)$ on~$\Zhat(1)_{\bar k}$
induces an action of $\Gal(\bar k/k)$ on the set of conjugacy classes of homomorphisms
$\hat\Z(1)_{\bar k} \to G$.  It therefore makes sense to consider the conjugates, under this action
of~$\Gal(\bar k/k)$,
of the algebraic local monodromy of~$\pi$ at~$b_i$.  (We recall that $b_i \in \P^1(k)$.)
This feature of the algebraic point of view
plays a crucial rôle in the rigidity method.
It is not visible on the topological side of the identification
of Remark~\ref{rmks:zhat1zhat}~(i),
except when $k=\R$. Indeed, in this case,
complex conjugation acts by multiplication by~$-1$ on $\hat\Z(1)_{\bar k}$,
hence maps the algebraic local monodromy of~$\pi$ at~$b_i$ to its ``inverse'';
while on the topological side, it sends a local counterclockwise loop around~$b_i$
to a local clockwise loop around~$b_i$, and hence again it maps the conjugacy
class of~$g_i$ to the conjugacy class of~$g_i^{-1}$.

(iii)
For an arbitrary field~$k$ of characteristic~$0$,
the choice of a topological generator of the procyclic group
$\Zhat(1)_{\bar k}$ should be thought of as an algebraic analogue of the choice of an orientation
of the punctured unit disc (an insight of Grothendieck,
see \cite[(2.1)]{deligneweil1}),
as is illustrated by Remark~\ref{rmks:zhat1zhat}~(ii).
\end{rmks}

\subsubsection{Factoring~$\rho$}

The
natural action of $\Gal(\bar \Q/\Q)$
 on the
finite set
$\Hom_{\text{cont}}\Big(\hat\Z(1)_{\bar \Q},G\Big)$
 of  continuous homomorphisms
 $\hat\Z(1)_{\bar \Q} \to G$ induces a continuous action
of $\Gal(\bar \Q/\Q)$
 on the quotient of this finite set by the  conjugation action of~$G$.
As
the functor
\begin{gather*}
\Big(\text{reduced $0$-dimensional varieties over~$\Q$}\Big)
\to \Bigg(\begin{array}{c}\text{finite sets endowed with a}\\\text{continuous action of $\Gal(\bar \Q/\Q)$}\end{array}\Bigg)
\end{gather*}
that sends a variety~$Z$ to the set~$Z(\bar \Q)$ is an equivalence of categories,
we can canonically associate with~$G$
 a reduced $0$\nobreakdash-dimensional variety~$\sC_G$ over~$\Q$
such that
\begin{align*}
\sC_G(\bar \Q)=\Hom_{\text{cont}}\Big(\hat\Z(1)_{\bar \Q},G\Big)/(\text{conjugation by $G$})\rlap,
\end{align*}
compatibly with the natural continuous actions of~$\Gal(\bar \Q/\Q)$ on both sides.
One then has a canonical 
 $\Gal(\bar k/k)$\nobreakdash-equivariant
identification
\begin{align*}
\sC_G(\bar k)=\Hom_{\text{cont}}\Big(\hat\Z(1)_{\bar k},G\Big)/(\text{conjugation by $G$})
\end{align*}
for any field~$k$ of characteristic~$0$, with algebraic closure~$\bar k$.

Given a $G$\nobreakdash-cover
$\pi:C \to \P^1_k$ over
a field~$k$ of characteristic~$0$
and a rational branch point
$b \in \P^1(k)$,
the algebraic local monodromy of~$\pi$ at~$b$
defined in~\textsection\ref{subsubsec:alglocalmon}
is thus an element of $\sC_G(k)$.  More generally, if $b \in \P^1_k$ is an
arbitrary
branch point of~$\pi$ (i.e.\ a closed point, not necessarily rational),
applying the definition of~\textsection\ref{subsubsec:alglocalmon}
to the $G$\nobreakdash-cover over~$k(b)$ obtained from~$\pi$ by scalar extension
from~$k$ to~$k(b)$
and to the rational branch point of this $G$\nobreakdash-cover
induced by~$b$, we obtain an element of $\sC_G(k(b))$, which we still
call the \emph{algebraic local monodromy of~$\pi$ at~$b$}.
Viewing this element as a morphism
$\Spec(k(b)) \to \sC_G$
and writing the branch locus $B \subset \P^1_k$ of~$\pi$
as the disjoint union, over the points $b \in B$, of the schemes $\Spec(k(b))$,
we thus obtain a morphism of varieties $B \to \sC_G$.
We call it the \emph{algebraic local monodromy morphism of~$\pi$}.

For any $r \geq 1$, we denote by
 $\sV^r \subset (\P^1_\Q \times \sC_G)^r$ the inverse image of~$\sU^r$ by the projection
$(\P^1_\Q \times \sC_G)^r \to (\P^1_\Q)^r$,
and by~$\sV_r$ the quotient of $\sV^r$ by the natural
action of the symmetric group~$S_r$.
We have thus produced an étale covering $\nu:\sV_r \to \sU_r$
of smooth varieties over~$\Q$.
For any field~$k$ of characteristic~$0$,
the set $\sV_r(k)$ can be identified with the set
of
reduced
$0$\nobreakdash-dimensional subvarieties of $\P^1_k \times \sC_G$ of degree~$r$
that map isomorphically to their image in~$\P^1_k$, or, what is the same,
to the set of
reduced
$0$\nobreakdash-dimensional subvarieties~$B$ of~$\P^1_k$ of degree~$r$
endowed with a morphism $B\to \sC_G$.

For any field~$k$ of characteristic~$0$, associating with each $G$\nobreakdash-cover
its branch locus together with
its algebraic local monodromy morphism
provides us with a map $\sH_{G,r}(k) \to \sV_r(k)$.
When $k=\Q(I)$ is the function field of a connected component~$I$ of~$\sH_{G,r}$,
the image of the generic point of~$I$
by this map gives rise to a rational map $I \dashrightarrow \sV_r$.
Being a rational map between étale coverings of the normal variety~$\sU_r$, it is in fact a morphism.
By letting~$I$ vary over all connected components of~$\sH_{G,r}$, we obtain, in this way,
a morphism $\rho':\sH_{G,r} \to \sV_r$ such that $\rho=\nu\circ\rho'$.

Let~$\Cl(G)$ denote the set of conjugacy classes of~$G$.
Remark~\ref{rmks:zhat1zhat}~(i) and (the proof of)
Corollary~\ref{cor:retgcovers} together
imply the following explicit description
of the complex fibres of~$\rho'$:

\begin{prop}
\label{prop:descrhop}
Let $B \subset \P^1_\C$ be a reduced $0$\nobreakdash-dimensional subvariety of degree~$r$.
Write $B=\{b_1,\dots,b_r\}$.
Let $C=(C_1,\dots,C_r)$ be an $r$\nobreakdash-tuple of nontrivial conjugacy classes of~$G$,
viewed as a map $B(\C) \to \sC_G(\C)$
via the identification $\sC_G(\C)=\Cl(G)$ of Remark~\ref{rmks:zhat1zhat}~(i).
Then the fibre of~$\rho'$
above the point of~$\sV_r(\C)$ defined by~$B$ and~$C$ can be identified with the
quotient $\ni^C_r(G)$ of the set of $r$\nobreakdash-tuples $(g_1,\dots,g_r) \in G^r$ satisfying the following
three conditions by the action of~$G$ on this set by simultaneous conjugation:
\begin{enumerate}
\item
$g_1\cdots g_r=1$;
\item $g_1,\dots,g_r$ generate~$G$;
\item
$g_i \in C_i$ for all $i \in \{1,\dots,r\}$.
\end{enumerate}
\end{prop}

\begin{proof}
This would be a direct consequence of the quoted references
if we knew that for a field~$k$ of characteristic~$0$ (here $k=\C$),
the map $\sH_{G,r}(k) \to \sV_r(k)$ induced by~$\rho'$
sends the isomorphism class of any $G$\nobreakdash-cover over~$k$ to
its algebraic local monodromy morphism.
By the definition of~$\rho'$, this is true
for those $G$\nobreakdash-covers whose branch locus is ``generic'',
in the sense that it is a point of $\sU_r(k)$ lying over the generic point of the variety~$\sU_r$
over~$\Q$ (i.e.\ when viewed as a morphism $\Spec(k) \to \sU_r$, its image is the generic point).
Thus, Proposition~\ref{prop:descrhop} holds when~$B$ is ``generic''.
As~$\rho'$ is a morphism between étale coverings of~$\sU_r$,
the validity
of Proposition~\ref{prop:descrhop} for arbitrary~$B$ follows.
\end{proof}

\subsubsection{Rational points of~$\sC_G$}

Viewing~$\bar \Q$ as a subfield of~$\C$,
Remark~\ref{rmks:zhat1zhat}~(i) also
induces an identification $\sC_G(\bar\Q)=\Cl(G)$.
Via this identification,
the natural action of $\Gal(\bar \Q/\Q)$ on~$\sC_G(\bar \Q)$ gives rise to the action
of $\Gal(\bar \Q/\Q)$
on~$\Cl(G)$ given by the formula $\sigma(g) = g^{-\chi(\sigma)}$ for $\sigma \in \Gal(\bar \Q/\Q)$
and $g \in G$, where $\chi:\Gal(\bar \Q/\Q)\twoheadrightarrow \hat\Z^*$ denotes the
cyclotomic character.
As a consequence,
the set $\sC_G(\Q)$ of
rational points of~$\sC_G$
gets identified with the set of \emph{rational} conjugacy classes of~$G$ in the following
sense:

\begin{defn}
\label{def:rationalconjugacy}
A conjugacy class~$C$ of a finite group~$G$ is \emph{rational} if
for every $g \in C$ and every integer $n \geq 1$
prime to the order of~$g$,
the element $g^n$ belongs to~$C$.
\end{defn}

\subsubsection{Rational points of~$\sV_r$}
\label{subsubsec:rpvr}

Here is a simple way to exhibit rational points of~$\sV_r$.
Let $b_1,\dots,b_r \in \P^1(\Q)$ be pairwise distinct.
Let $B=\{b_1,\dots,b_r\}$.
The rational points of~$\sV_r$ lying above the rational point of~$\sU_r$ defined by~$B$
are exactly the $r$\nobreakdash-tuples of rational points of~$\sC_G$, i.e.\ they are
the $r$\nobreakdash-tuples of rational conjugacy
classes of~$G$.

\subsubsection{Rational points of~$\sH_{G,r}$}

The morphism $\rho':\sH_{G,r} \to \sV_r$ is an étale covering. As such, it has a degree, which is
a locally constant function on~$\sV_r$.  This function is not constant on~$\sV_r$ in general---unlike
the degree of~$\rho$, which is constant, equal to the cardinality of~$\ni_r(G)$, as we
have seen in~\textsection\ref{subsec:hurwitzspaces}.
A connected component of~$\sV_r$ over which~$\rho'$ has degree~$1$,
i.e.\ over which~$\rho'$ restricts to an isomorphism,
is said to be \emph{rigid}.

It is a trivial but fruitful observation, which forms the basis of the rigidity method,
that the existence of a rational point
of a rigid connected component of~$\sV_r$ implies the existence
of a rational point of~$\sH_{G,r}$.

By Proposition~\ref{prop:descrhop}, the rigidity of a connected component
can be verified by computing
the set $\ni^C_r(G)$ associated with its complex points.
This motivates
the following definition:

\begin{defn}
An $r$\nobreakdash-tuple $C=(C_1,\dots,C_r)$ of nontrivial conjugacy classes of~$G$
is \emph{rigid} if 
the set $\ni^C_r(G)$ defined in Proposition~\ref{prop:descrhop}
has cardinality~$1$.
\end{defn}

\subsubsection{Summing up}

We thus arrive at a down-to-earth condition that implies that the
rational points of~$\sV_r$ constructed in~\textsection\ref{subsubsec:rpvr}
can be lifted to rational points of~$\sH_{G,r}$.

\begin{thm}
\label{th:rigidity}
Let~$G$ be a finite group with trivial centre.
Let $r\geq 1$ be an integer.
If there exists a rigid $r$\nobreakdash-tuple
 $C=(C_1,\dots,C_r)$ of nontrivial rational conjugacy classes of~$G$,
then for any pairwise distinct $b_1,\dots,b_r \in \P^1(\Q)$,
the $\Q$\nobreakdash-point of~$\sU_r$ defined by $\{b_1,\dots,b_r\}$ can be lifted to a
$\Q$\nobreakdash-point of~$\sH_{G,r}$.  In particular, the regular inverse Galois problem admits a positive
solution for~$G$ over~$\Q$.
\end{thm}

Theorem~\ref{th:rigidity} represents the base case of the rigidity method.  It admits many variants;
for instance, one can allow non-rational branch points. (Pro: this weakens
the condition that the prescribed conjugacy classes be rational; con: these conjugacy classes cannot be chosen independently of one another
any longer.) One could also work over a number field other than~$\Q$, which would simultaneously
weaken
 the conclusion of the theorem
and
 the rationality assumption on the conjugacy classes~$C_i$.

Even just the above base case is already unreasonably effective: the hypothesis of
Theorem~\ref{th:rigidity}
has been shown to be satisfied,
 with $r=3$, for at least~$10$ of the~$26$ sporadic simple groups,
including the monster (by Thompson) and the baby monster (by Malle and Matzat);
see \cite[Chapter~II, \textsection9]{mallematzat}.
As another example, the variant of Theorem~\ref{th:rigidity} in which $r=3$ and~$b_1$ is rational
while~$b_2$ and~$b_3$ are conjugate
quadratic points of the projective line can be applied to $\PSL_2(\Fp)$ for all primes~$p$ such that~$2$ or~$3$
is not a square modulo~$p$ (see \cite[\textsection8.3.3]{serretopics}),
thus recovering the positive answer
to the regular inverse
Galois problem over~$\Q$ for this infinite family of groups that had been obtained
by Shih using modular curves rather than the rigidity
method.

When the rigidity method is applicable, it is in principle possible to deduce from it an explicit polynomial
that realises the desired regular Galois extension of~$\Q(t)$
(see \cite[Chapter~II, \textsection9]{mallematzat}).  This has some limits in practice (e.g.\ for the
monster group, the degree of the polynomial cannot have less than~$20$ digits) but it leads to
interesting computational challenges (see e.g.\ \cite{barthwenz}).

\section{Grunwald's problem and the Brauer--Manin obstruction}
\label{sec:3}

\subsection{Looking for rational points}
\label{subsec:lookingforratp}

Despite its successes, the rigidity method,
discussed in~\textsection\ref{subsec:rigidity},
often fails to be applicable.
For instance, it fails, in general, for $p$\nobreakdash-groups;
indeed, the regular inverse Galois problem
is still open for most $p$\nobreakdash-groups over~$\Q$, even though the inverse Galois problem itself is known to have a positive
answer, over any number field,
for all $p$\nobreakdash-groups---and more generally, for all solvable groups, by a celebrated theorem of Shafarevich
(see \cite[Chapter~IX, \textsection6]{neukirchschmidtwingberg}).
In other words,
after learning, in~\textsection\ref{sec:1},
that the quotient variety $\A^n_k/G$ can fail to be rational or even retract rational,
we now find ourselves unable, at least in practice,
to salvage the Hilbert--Noether method by constructing rational curves in~$\A^n_k/G$
over which to apply Hilbert's irreducibility theorem,
as envisaged at the beginning of~\textsection\ref{subsec:rigstatement}.  This leads us to our next question:
letting
$Y \subseteq \A^n_k$ denote the locus where~$G$ acts freely,
can we directly construct rational points of $Y/G$ above which the fibre of the quotient map $Y\to Y/G$
is irreducible?
An approach put forward by Colliot-Thélène consists in noting that
Ekedahl's Theorem~\ref{th:ekedahl} reduces this
question, in full generality, to the problem of finding rational points on~$Y/G$
subject to certain weak approximation conditions.
In particular, if the variety~$Y/G$ satisfies weak approximation
off a finite set of places of~$k$, then the inverse Galois problem has a positive answer for~$G$ over~$k$.
Such a weak approximation property can be proved unconditionally
in some cases; for instance, under the assumptions of the following
remarkable theorem from~\cite{neukirch-solvable}, in which no place of~$k$ is excluded:

\begin{thm}[Neukirch]
\label{th:neukirch}
Let~$k$ be a number field.
Let~$G$ be a finite solvable group, acting linearly on~$\A^n_k$ for some $n \geq 1$.
Let $Y \subseteq \A^n_k$ be the locus where~$G$ acts freely.
Let $X=Y/G$.
Assume that the order of~$G$ and the number of roots of unity contained in~$k$ are coprime.
Then~$X$ satisfies weak approximation, i.e.\ the set $X(k)$ is dense in $X(k_\Omega)$.
In particular, Grunwald's problem admits a positive
answer for~$G$ over~$k$, for any finite subset $S \subset \Omega$.
\end{thm}

We recall that Grunwald's problem is Problem~\ref{pb:grunwald}.
Without the assumption on the order of~$G$, the conclusion of this theorem can fail,
as we have seen in Proposition~\ref{prop:wang}.

The validity of the weak approximation property is a problem of general interest that makes sense,
and has been studied, for arbitrary smooth varieties.  As we shall now explain, the tools that have been developed
for its study on arbitrary smooth varieties turn out to be useful also in the special case of the quotient~$Y/G$.

\subsection{Brauer--Manin obstruction}

A general mechanism, introduced by Manin~\cite{maninicm}
and now called the \emph{Brauer--Manin obstruction},
explains, in some cases, why
the weak approximation property fails for certain varieties over number fields.
Let us recall it briefly. (For details,
see \cite[\textsection13.3]{ctskobook}.)
Let~$X$ be a smooth variety over a number field~$k$.
Let $\Br_\nr(X)$ denote its unramified Brauer group
(see~\textsection\ref{subsubsec:counterexamplesoverc}).
We let $X(k_\Omega)=\prod_{v \in \Omega}X(k_v)$ and endow this set with the product of the $v$\nobreakdash-adic topologies.

The \emph{Brauer--Manin set}
$X(k_{\Omega})^{\Br_{\nr}(X)}$
is the set of all
families $(x_v)_{v \in \Omega} \in X(k_\Omega)$
such that $\sum_{v \in \Omega} \inv_v \beta(x_v)=0$  for all $\beta \in \Br_{\nr}(X)$.
Here $\beta(x_v) \in \Br(k_v)$ denotes the evaluation of~$\beta$ at~$x_v$,
and $\inv_v:\Br(k_v) \hookrightarrow \Q/\Z$ is the invariant map from local class
field theory.
(To make sense of the sum, one checks that only finitely many of its terms are nonzero.)
Manin's fundamental observation is that the image of the diagonal embedding $X(k) \hookrightarrow X(k_\Omega)$
is contained in the Brauer--Manin set, as a consequence
of
the reciprocity law
of global class field theory.
Thus, we have a sequence of
 inclusions
\begin{align}
X(k) \subseteq X(k_\Omega)^{\Br_\nr(X)} \subseteq X(k_\Omega)\rlap.
\end{align}
As $X(k_\Omega)^{\Br_\nr(X)}$ is a closed subset of $X(k_\Omega)$,
the
weak approximation property,
i.e.\ the density of~$X(k)$ in $X(k_\Omega)$, can hold only if $X(k_\Omega)^{\Br_\nr(X)}=X(k_\Omega)$.
When this last equality fails,
one says that there is a Brauer--Manin obstruction to weak approximation
on~$X$.

\subsection{Reinterpreting the Grunwald--Wang theorem}
\label{subsec:reinterpretinggrunwaldwang}

Let us come back to the variety $X=Y/G$ considered in~\textsection\ref{subsec:lookingforratp}:
$G$ is a subgroup of~$S_n$, which acts on~$\A^n_k$ by permuting the coordinates, and $Y \subseteq \A^n_k$ is the locus
where~$G$ acts freely.
As Grunwald's problem has a negative answer
for
 $G=\Z/8\Z$
and
 $k=\Q$
(see Proposition~\ref{prop:wang})
and as $Y\to X$ is a versal $G$\nobreakdash-torsor (see Example~\ref{ex:noetherisversal}),
the variety~$X$ cannot satisfy the weak approximation property in this case, according
to Proposition~\ref{prop:versalrationalimpliesgrunwald}.  Hence a natural question: is there a Brauer--Manin
obstruction to weak approximation on~$X$ when $G=\Z/8\Z$ and $k=\Q$?
The answer is yes
by the following theorem,
which states, more precisely,
that the weak approximation property on~$X$ is fully controlled by the Brauer--Manin obstruction
as soon as~$G$ is abelian.

\begin{thm}[Voskresenski\u{\i}, Sansuc]
\label{th:vosksansuc}
Let~$k$ be a number field.
Let~$G$ be a finite abelian group, acting linearly on~$\A^n_k$ for some $n \geq 1$.
Let $Y \subseteq \A^n_k$ be the locus where~$G$ acts freely.
Let $X=Y/G$.
The set $X(k)$ is dense in $X(k_\Omega)^{\Br_\nr(X)}$.
\end{thm}

Theorem~\ref{th:vosksansuc} can be found in the literature by combining
\cite[§7.2, Theorem]{voskbook}
with
\cite[Corollaire~8.13]{sansuclinear}.
We shall explain a proof of it
in~\textsection\ref{subsec:sketch} below.

Returning to an arbitrary finite group~$G$ and keeping~$k$, $Y$ and~$X$ as above,
it is a general fact
that the density of~$X(k)$ in~$X(k_\Omega)^{\Br_\nr(X)}$ implies the existence of a finite
subset $S_0 \subset \Omega$ such that~$X$ satisfies weak approximation off~$S_0$.
When such an~$S_0$ exists,
a refinement of the arguments underlying the proof of
Proposition~\ref{prop:versalrationalimpliesgrunwald} 
leads to the following conclusion
(a point of view advocated in~\cite{chernousov} and in~\cite[\textsection1.2]{harariquelques}):
fully solving Grunwald's problem for~$G$ over~$k$ is in fact equivalent to describing
the closure of~$X(k)$ inside $X(k_\Omega)$.
Thus, the Grunwald--Wang theorem, which indeed
fully solves Grunwald's problem when~$G$ is abelian, can now be
viewed, in retrospect, as the combination of Theorem~\ref{th:vosksansuc} with an explicit
computation of the Brauer--Manin set $X(k_\Omega)^{\Br_\nr(X)}$ when~$G$ is abelian.

\subsection{Rationally connected varieties}
\label{subsec:rcvar}

Whether or not the abelianness hypothesis on~$G$ can be removed from Theorem~\ref{th:vosksansuc} is a fundamental
open question.
When~$X$ is an arbitrary smooth variety possessing a rational point, the set~$X(k)$
cannot be expected to be dense
in $X(k_\Omega)^{\Br_{\nr}(X)}$ without strong assumptions on the geometry of~$X$;
for instance,
Lang's conjectures predict that
for~$d$ and~$N$ such that $d-2 \geq N \geq 4$,
this density should fail
for all smooth hypersurfaces of degree~$d$ in~$\P^N$
that have a rational point (see \cite[Appendix~A]{poonenvoloch}).
The variety $X=Y/G$ that we have been considering
in~\textsection\ref{subsec:lookingforratp}
and in~\textsection\ref{subsec:reinterpretinggrunwaldwang},
despite not being
geometrically rational for an arbitrary finite group~$G$
(see \textsection\ref{subsubsec:counterexamplesoverc}), still has a reasonably
tame geometry: it is unirational and therefore belongs to the class of \emph{rationally connected varieties} according
to the following definition\footnote{To be precise, Definition~\ref{def:rc} coincides with the standard definition
(found, e.g., in \cite[Chapter~IV]{kollarbook}) only when~$X$ is proper.}.

\begin{defn}[Campana, Kollár, Miyaoka, Mori]
\label{def:rc}
A smooth variety~$X$ over a field~$k$ is said to be
\emph{rationally connected} if for any algebraically closed field extension $K/k$
and any two general $K$\nobreakdash-points $x_0,x_1 \in X(K)$, there exists a rational map $f:\A^1_K \dashrightarrow X_K$ over~$K$,
defined in a neighbourhood of~$0$ and~$1$, such that $f(0)=x_0$ and $f(1)=x_1$.  (``General'' means that the set
of pairs $(x_0,x_1)$ satisfying the stated condition contains a dense Zariski open subset of $X(K) \times X(K)$.)
\end{defn}

Theorem~\ref{th:vosksansuc} conjecturally extends to all smooth rationally connected varieties:

\begin{conj}[Colliot-Thélène]
\label{conj:ct}
Let~$X$ be a smooth, rationally connected variety, over a number field~$k$.
The set $X(k)$ is dense in $X(k_\Omega)^{\Br_\nr(X)}$.
\end{conj}

A number of known results towards this conjecture are listed in~\cite{wittenbergslc}.
Conjecture~\ref{conj:ct} would imply that all smooth rationally connected varieties satisfy
weak approximation off a finite set of places (see e.g.\ \cite[Remarks~2.4~(i)--(ii)]{wittenbergslc}).
In particular, it would imply a positive answer to the inverse Galois problem in general,
by Theorem~\ref{th:ekedahl} applied to the torsor of Example~\ref{ex:noetherisversal}.

A list of groups~$G$ of small order for which Conjecture~\ref{conj:ct} is still open
for the variety $X=Y/G$
appearing in Example~\ref{ex:noetherisversal}
can be found in~\cite{boughattasneftin}.

\subsection{Determining the Brauer--Manin set}
\label{subsec:determiningbm}

As discussed in~\textsection\ref{subsec:reinterpretinggrunwaldwang}, Conjecture~\ref{conj:ct} in the case $X=Y/G$
would, more precisely,
reduce
Grunwald's problem for~$G$ over~$k$
to the computation of the Brauer--Manin set
of~$X$.
Even partial knowledge of the Brauer--Manin set can lead
to concrete results, as the following theorem illustrates:

\begin{thm}[Lucchini Arteche \cite{lucchiniunramifiedbrauer}]
\label{th:lucchiniarteche}
Let~$k$ be a number field and~$G$ be a finite group acting linearly on~$\A^n_k$. Let $Y \subseteq \A^n_k$ be
the locus where~$G$
acts freely.  Let $X=Y/G$.  Let~$S_0$ be the set of finite places of~$k$ that divide the order of~$G$.
If~$X$ satisfies Conjecture~\ref{conj:ct}, then
Grunwald's problem admits a positive
answer for~$G$ over~$k$, for any finite subset $S \subset \Omega$ disjoint from~$S_0$.
\end{thm}

The proof of Theorem~\ref{th:lucchiniarteche}  consists in studying
the evaluation of unramified Brauer classes at the local points of~$X$,
so as
to deduce, from the density of~$X(k)$ in~$X(k_\Omega)^{\Br_\nr(X)}$,
that~$X$ satisfies weak approximation off~$S_0$;
Proposition~\ref{prop:versalrationalimpliesgrunwald} then yields the desired statement.

The complete determination of~$X(k_\Omega)^{\Br_\nr(X)}$, for~$X$ as in the statement of Theorem~\ref{th:lucchiniarteche},
is in general a difficult task.
The case of a metabelian group~$G$ is investigated in \cite{demeioramified}.
In general, even the computation of~$\Br_\nr(X)$ itself is a delicate problem.
Over an algebraic closure~$\bar k$ of~$k$, one can apply Bogomolov's formula (Theorem~\ref{th:bogomolovformula}).
If the unramified Brauer group of~$X_{\bar k}$ turns out to be nontrivial, one has
to find out which classes of the finite group $\Br_\nr(X_{\bar k})$
are invariant under $\Gal(\bar k/k)$, and to determine
the image
of the natural map $\Br_{\nr}(X)\to \Br_{\nr}(X_{\bar k})^{\Gal(\bar k/k)}$; there is no general recipe for carrying this out.
The kernel of the latter map, on the other hand, is now well understood:
its quotient by the image
of the natural map $\Br(k) \to \Br_\nr(X)$
is finite and is described by a formula due to Harari \cite[Proposition~4]{harariquelques}.
In the most favourable cases, the combination of these formulae can lead to the conclusion
that the natural map $\Br(k) \to \Br_\nr(X)$ is onto, so that $X(k_\Omega)^{\Br_\nr(X)}=X(k_\Omega)$
and Grunwald's problem is then expected to have a positive solution for~$G$ over~$k$ with no restriction
on the finite subset $S \subset \Omega$. See \cite[Remarque~7]{demarchebrnr} for a concrete example.
In a different direction,
by adapting
the proof of Bogomolov's formula
to non-algebraically closed ground fields,
Colliot-Thélène~\cite[Corollaire~5.7]{ctbrnslng} showed that the natural map $\Br(k) \to \Br_\nr(X)$
is onto, and hence
that $X(k_\Omega)^{\Br_\nr(X)}=X(k_\Omega)$,
 whenever the order of~$G$ and the number of roots of unity contained in~$k$ are coprime,
which explains, a posteriori, why the Brauer--Manin obstruction plays no rôle in Neukirch's Theorem~\ref{th:neukirch}.

\subsection{Supersolvable groups}
\label{subsec:ss}

A finite group~$G$ is said to be \emph{supersolvable} if there exists a filtration
$1 = G_0 \subseteq G_1 \subseteq \dots \subseteq G_n = G$ such
that each~$G_i$ is a normal subgroup of~$G$ and each successive
quotient $G_{i+1}/G_i$ is cyclic.  All nilpotent groups
(in particular, all $p$\nobreakdash-groups) are supersolvable.
We proved, in \cite{hwzceh}, that Theorem~\ref{th:vosksansuc} generalises to such groups:

\begin{thm}[Harpaz, W.]
\label{thm:hwzceh}
Let~$k$ be a number field.
Let~$G$ be a finite supersolvable group, acting linearly on~$\A^n_k$ for some $n \geq 1$.
Let $Y \subseteq \A^n_k$ be the locus where~$G$ acts freely.
Let $X=Y/G$.
The set $X(k)$ is dense in $X(k_\Omega)^{\Br_\nr(X)}$.
\end{thm}

A positive answer to the inverse Galois problem
for supersolvable groups results from this
(via Theorem~\ref{th:ekedahl})
but had already been established---more generally, for solvable groups---by Shafarevich, as mentioned in~\textsection\ref{subsec:lookingforratp}.
As discussed in \textsection\textsection\ref{subsec:reinterpretinggrunwaldwang}--\ref{subsec:determiningbm},
Theorem~\ref{thm:hwzceh} refines this positive
answer by bringing information about Grunwald's problem.

It may be that the strategy underlying the proof of Theorem~\ref{thm:hwzceh} can be extended to all solvable groups.
It will not, however, be of any help with
non-abelian simple groups; in fact,
to this day, no approach is known towards
 Grunwald's problem for non-abelian simple groups (with the exception of~$A_5$ and~$\PSL_2(\F_7)$,
for which Noether's problem itself has a positive answer; see Example~\ref{ex:noethera5}, \cite{mestrepsl2f7} and
Corollary~\ref{cor:noetherimpliesgrunwald}).

\subsection{Descent in a nutshell}

Theorem~\ref{thm:hwzceh} can be seen as a direct application of
a general tool that is useful for proving cases
of Conjecture~\ref{conj:ct}, the so-called ``descent'' method.
We now briefly discuss it.  We shall illustrate it by proving
Theorem~\ref{th:vosksansuc} in~\textsection\ref{subsec:sketch}, before indicating its applicability
to other variants of the inverse Galois problem in~\textsection\ref{subsec:prescribednorms}.

To get started, we need to extend the notion of $G$\nobreakdash-torsor from the case where~$G$
is a finite abstract group (Definition~\ref{def:torsor})
to the case where~$G$ is an algebraic group (i.e.\ a group scheme over a field,
possibly disconnected or of positive dimension).

\begin{defn}
\label{def:torsoralg}
Let $\pi:Y\to X$ be a surjective morphism between smooth varieties over a field~$k$ of characteristic~$0$,
with algebraic closure~$\bar k$.
Let~$G$ be an algebraic group over~$k$, acting on~$Y$ in such a way that~$\pi$ is $G$\nobreakdash-equivariant
(for the trivial action of~$G$
on~$X$).
We say that~$\pi$ is a \emph{$G$\nobreakdash-torsor},
or that~$Y$ is a $G$\nobreakdash-torsor over~$X$,
if~$\pi$ is smooth and~$G(\bar k)$
acts simply transitively on the fibres of the map $Y(\bar k)\to X(\bar k)$ induced by~$\pi$.
\end{defn}

Unless otherwise specified, we now let~$k$ denote an arbitrary field of characteristic~$0$.
(For the correct definition of a torsor without this assumption,
see \cite[Definition~2.2.1]{skobook}.)
As usual, by a \emph{$G$\nobreakdash-torsor over~$k$} we shall mean a $G$\nobreakdash-torsor over~$\Spec(k)$.
When $\pi:Y\to X$ is a $G$\nobreakdash-torsor, the morphism~$\pi$ identifies~$X$ with the categorical
quotient~$Y/G$ (see \cite[Proposition~0.2, Proposition~0.1]{mumfordgit}).

\begin{example}
\label{ex:hilbert90}
Hilbert's Theorem~90, which we encountered in Example~\ref{ex:noetherisversal},
is equivalent to the following statement: for any integer $n\geq 1$,
any $\GL_n$\nobreakdash-torsor over~$k$ is isomorphic to~$\GL_n$.
As an easy consequence, any $\SL_n$\nobreakdash-torsor over~$k$ is isomorphic to~$\SL_n$.
\end{example}

\begin{defn}
Let~$X$ be a smooth variety over~$k$. Let~$G$ be an algebraic group over~$k$.
Let $\pi:Y\to X$ be a $G$\nobreakdash-torsor.
The \emph{twist} of~$Y$ by a $G$\nobreakdash-torsor~$P$ over~$k$ is the quotient
\begin{align*}
_PY=(P \times Y)/G
\end{align*}
of~$P \times Y$
by the diagonal action of~$G$, endowed with the natural morphism $_P\pi:{}_PY\to X$ induced by
the second projection $P\times Y\to Y$ and the identification $Y/G=X$.  (We are not claiming that $_P\pi$ is a
$G$\nobreakdash-torsor.  This is true when~$G$ is commutative, not in general.  This point, which will be irrelevant
for us, is discussed in \cite[p.~21]{skobook}.)
\end{defn}

The gist of the descent method is summarised in the following conjecture.

\begin{conj}
\label{conj:descent}
Let~$X$ be a smooth variety over a number field~$k$.
Let~$G$ be a
linear algebraic group over~$k$.
Let $\pi:Y\to X$ be a $G$\nobreakdash-torsor,
with~$Y$ rationally connected.
Assume that for every twist~$Y'$ of~$Y$ by a $G$\nobreakdash-torsor over~$k$,
the set $Y'(k)$ is dense in $Y'(k_\Omega)^{\Br_{\nr}(Y')}$.
Then the set $X(k)$ is dense in $X(k_\Omega)^{\Br_\nr(X)}$.
\end{conj}

Conjecture~\ref{conj:descent}, and the first significant cases in which it was established,
appeared in a series of works by Colliot-Thélène and Sansuc.  See \cite{skobook} for an account.
We content ourselves with mentioning the following positive result\footnote{More generally,
after the writing of these notes,
Conjecture~\ref{conj:descent} was proved
by Nguy\~{\^e}n to be true as soon as~$G$ is connected; see~\cite{nguyendescent},
which builds on~\cite[\textsection2]{hwzceh} and on Borovoi's abelianisation technique.}, which in
this form can be found
in \cite[Corollaire~2.2]{hwzceh}.

\begin{thm}[Colliot-Thélène, Sansuc, Harpaz, W.]
\label{th:descenttori}
Conjecture~\ref{conj:descent} holds true if~$G$ is a torus (i.e.\ an algebraic group such that
$G_{\bar k} \simeq \Gmbark \times \dots \times \Gmbark$).
\end{thm}

\subsection{Sketch of proof of Theorem~\ref{th:vosksansuc}}
\label{subsec:sketch}

We shall deduce Theorem~\ref{th:vosksansuc} from
Theorem~\ref{th:descenttori}.  (More precisely, descent will be applied
to a geometrically rational variety;
Theorem~\ref{th:descenttori}
in the case of such varieties is due to Colliot-Thélène and Sansuc alone, see~\cite{ctsandescent2}.)

Before starting the proof of Theorem~\ref{th:vosksansuc}, let us slightly change notation.
We now fix an embedding $G \hookrightarrow \SL_n(k)$ for some $n\geq 1$, let~$Y$
be the algebraic group~$\SL_n$ over~$k$, and let~$G$ act on~$Y$ by right multiplication.  As the resulting variety $X=Y/G$
is stably birationally equivalent to the variety~$X$ of Theorem~\ref{th:vosksansuc} (see Remark~\ref{rmk:noname}),
and as the density of~$X(k)$ in $X(k_\Omega)^{\Br_\nr(X)}$ is a stable birational
invariant among smooth, rationally connected varieties (see \cite[Proposition~6.1~(iii)]{cps}
and \cite[Remark~2.4~(ii)]{wittenbergslc}),
this change of notation is harmless.

Recall that~$G$, by assumption, is a finite abelian group.
Let us view it as a constant algebraic group over~$k$.
It is easy to see that~$G$ fits into a short exact sequence
\begin{align}
1 \to G \to T \to Q \to 1
\end{align}
of algebraic groups over~$k$, where~$T$ and~$Q$ are tori
and~$Q$ is quasi-trivial, i.e.\ the character group $\Hom(Q_{\bar k},\Gmbark)$ of~$Q$
admits a basis over~$\Z$ that is stable under the action of~$\Gal(\bar k/k)$ (see \cite[Proposition~4.2.1]{serretopics}).
Letting~$G$ act on~$T$ by translation,
we now consider the quotient $W=(\SL_n \times T)/G$
of~$\SL_n \times T$ by the diagonal action,
together with the morphism $\pi:W\to X=\SL_n/G$
induced by the first projection.

The action of~$T$ on~$\SL_n \times T$ by multiplication on the second factor induces an action
of~$T$ on~$W$, with respect to which~$\pi$ is a $T$\nobreakdash-torsor.
According to Theorem~\ref{th:descenttori} applied to~$\pi$,
it will suffice,
in order to complete the proof of Theorem~\ref{th:vosksansuc},
to show that for every $T$\nobreakdash-torsor~$P$ over~$k$,
the variety~$_PW$ satisfies Conjecture~\ref{conj:ct}.

We observe that $_PW=(\SL_n \times P)/G$, that the morphism $p:{}_PW \to P/G$
induced by the second projection is an $\SL_n$\nobreakdash-torsor
(with respect to the action of~$\SL_n$ on~$_PW$ coming from its action on
$\SL_n \times P$ by left multiplication on the first factor), and that
$P/G$ is a $Q$\nobreakdash-torsor over~$k$.
By Hilbert's Theorem~90 (see Example~\ref{ex:hilbert90}),
the generic fibre of~$p$ is isomorphic to~$\SL_n$, in particular it is rational.
It also follows from Hilbert's Theorem~90 (case $n=1$ of
Example~\ref{ex:hilbert90}), combined with Shapiro's lemma, that
any torsor under a quasi-trivial torus over~$k$ is rational; in particlar, the variety~$P/G$
is rational.  These two remarks
imply that~$_PW$  is itself rational over~$k$.   Thus, it satisfies Conjecture~\ref{conj:ct}
for trivial reasons, and Theorem~\ref{th:vosksansuc} is proved.

\subsection{Supersolvable descent}

The ideas sketched in~\textsection\ref{subsec:sketch} are a starting point
for the proof of the following theorem, established in \cite[Corollary~3.3]{hwsupersolvable}.

\begin{thm}[Harpaz, W.]
\label{th:descentsupersolvable}
Conjecture~\ref{conj:descent} holds true if~$G$ is finite and supersolvable.
\end{thm}

Here ``supersolvable'' means that $G(\bar k)$ is supersolvable in the sense
recalled at the beginning of~\textsection\ref{subsec:ss}, except that the
filtration is now required, in addition, to be stable under the action
of~$\Gal(\bar k/k)$ on~$G(\bar k)$, in case this action is not trivial.

Theorem~\ref{th:descentsupersolvable} implies
Theorem~\ref{thm:hwzceh}: in the notation introduced at the beginning of~\textsection\ref{subsec:sketch},
it suffices to apply
Theorem~\ref{th:descentsupersolvable} to the $G$\nobreakdash-torsor $\SL_n \to \SL_n/G$
and to note that any twist of~$\SL_n$ by a $G$\nobreakdash-torsor over~$k$
is an $\SL_n$\nobreakdash-torsor
(through left multiplication),
hence is isomorphic to~$\SL_n$, by Hilbert's Theorem~90,
hence is rational and satisfies Conjecture~\ref{conj:ct}.

\begin{rmk}
Theorem~\ref{th:descentsupersolvable} implies in particular that Conjecture~\ref{conj:descent} holds true
for $G=\Z/2\Z$.  We recall that
 Conjecture~\ref{conj:descent}
assumes that the variety~$Y$ appearing in
its statement is rationally connected.
As was pointed out to the author by M\d{a}nh Linh Nguy\~{\^e}n,
this assumption cannot be dropped,
even when $G=\Z/2\Z$.
Indeed, there exist examples of double covers $Y\to X$, where~$Y$ is a~$K3$ surface and~$X$ is an Enriques surface over $k=\Q$,
such that $X(k)$ and the sets $Y'(k_\Omega)^{\Br_{\nr}(Y')}$ are all empty, while $X(k_\Omega)^{\Br_{\nr}(X)}$
is not
(see \cite[Theorem~1.2]{bbmpv}).
\end{rmk}

\subsection{Prescribed norms}
\label{subsec:prescribednorms}

Our last theorem is an application of supersolvable descent to a variant of the inverse Galois problem of
a slightly different flavour, meant to demonstrate the flexibility of descent as a tool.

\newcommand{\citehss}{\cite[Theorem~4.16]{hwsupersolvable}}
\begin{thm}[\citehss]
\label{th:normsfrom}
Let~$G$ be a finite group.
Let~$k$ be a number field.
Let $\alpha_1,\dots,\alpha_m \in k^*$.
If~$G$ is supersolvable, there exists a Galois field extension~$K$ of~$k$
such that $\Gal(K/k)\simeq G$ and $\alpha_1,\dots,\alpha_m \in N_{K/k}(K)$.
\end{thm}

The idea of the proof is to construct, in a formal and explicit way,
a $G$\nobreakdash-torsor $\pi:Y\to X$
together with invertible functions $\beta_1,\dots,\beta_m$ on~$Y$ whose norms along~$\pi$ are equal
to the constant invertible functions~$\alpha_1,\dots,\alpha_m$ on~$X$.  Namely, say $m=1$ for simplicity,
embed~$G$ into~$\SL_n(k)$ and consider the subvariety~$Y$ of $\SL_n \times \prod_{g \in G} \Gm$
consisting of all $(s,(t_g)_{g\in G})$ such that $\prod_{g \in G} t_g=\alpha_1$;
the invertible function~$\beta_1$ on~$Y$
given by projection onto the~$\Gm$ factor corresponding to $1 \in G$
 has the required norm.
One then checks that the twists of~$Y$ satisfy Conjecture~\ref{conj:ct} (despite not being rational in general, even
when they have a rational point and~$G$ is assumed abelian), so that Theorem~\ref{th:descentsupersolvable}
implies the validity of Conjecture~\ref{conj:ct}, and hence of weak approximation off a finite set of places,
for~$X$.  As $X(k)\neq\emptyset$ (indeed even $Y(k)\neq\emptyset$),
it follows, by Theorem~\ref{th:ekedahl}, that there exists $x \in X(k)$ such that~$\pi^{-1}(x)$ is irreducible,
i.e.\ gives rise to a Galois field extension~$K/k$ with group~$G$.
Restricting~$\beta_1,\dots,\beta_m$ to~$\pi^{-1}(x)$ yields elements of~$K^*$ with norms~$\alpha_1,\dots,\alpha_m$.

In the case where~$G$ is abelian, Theorem~\ref{th:normsfrom} was first shown to hold
by Frei, Loughran and Newton~\cite{freiloughrannewton}, who established an asymptotic estimate for the number of such~$K$.

\bibliographystyle{amsalpha}
\bibliography{parkcity}

\newcommand{\etalchar}[1]{$^{#1}$}
\providecommand{\bysame}{\leavevmode\hbox to3em{\hrulefill}\thinspace}
\providecommand{\MR}{\relax\ifhmode\unskip\space\fi MR }
\providecommand{\MRhref}[2]{%
  \href{http://www.ams.org/mathscinet-getitem?mr=#1}{#2}
}
\providecommand{\href}[2]{#2}
\begin{thebibliography}{BCTSSD85}

\bibitem[AD00]{amorosodvornicich}
F.~Amoroso and R.~Dvornicich, \emph{A lower bound for the height in abelian
  extensions}, J. Number Theory \textbf{80} (2000), no.~2, 260--272.

\bibitem[AM72]{artinmumford}
M.~Artin and D.~Mumford, \emph{Some elementary examples of unirational
  varieties which are not rational}, Proc. Lond. Math. Soc. (3) \textbf{25}
  (1972), 75--95.

\bibitem[AT09]{artintate}
E.~Artin and J.~Tate, \emph{Class field theory}, AMS Chelsea Publishing,
  Providence, RI, 2009, reprinted with corrections from the 1967 original.

\bibitem[BBM{\etalchar{+}}16]{bbmpv}
F.~Balestrieri, J.~Berg, M.~Manes, J.~Park, and B.~Viray, \emph{Insufficiency
  of the {Brauer}-{Manin} obstruction for rational points on {Enriques}
  surfaces}, Directions in number theory. Proceedings of the 2014 {W}omen in
  {N}umbers, WIN3 workshop. Banff, Alberta, Canada, April 21--25, 2014, Cham:
  Springer, 2016, pp.~1--31.

\bibitem[BCTSSD85]{bctssd}
A.~Beauville, J.-L. Colliot-Th\'{e}l\`ene, J.-J. Sansuc, and
  P.~Swinnerton-Dyer, \emph{Variétés stablement rationnelles non
  rationnelles}, Ann. Math. (2) \textbf{121} (1985), 283--318.

\bibitem[BF03]{berhuyfavi}
G.~Berhuy and G.~Favi, \emph{Essential dimension: {A} functorial point of view
  (after {A}. {Merkurjev})}, Doc. Math. \textbf{8} (2003), 279--330.

\bibitem[BN24]{boughattasneftin}
E.~Boughattas and D.~Neftin, \emph{The {G}runwald problem and homogeneous
  spaces with non-solvable stabilisers}, to appear in Annales de l'Institut
  Fourier, \href{https://arxiv.org/abs/2404.09874}{arXiv:2404.09874}, 2024.

\bibitem[Bog88]{bogomolovbrnr}
F.~A. Bogomolov, \emph{The {Brauer} group of quotient spaces by linear group
  actions}, Math. USSR, Izv. \textbf{30} (1988), no.~3, 455--485.

\bibitem[BR97]{buhlerreichstein}
J.~Buhler and Z.~Reichstein, \emph{On the essential dimension of a finite
  group}, Compos. Math. \textbf{106} (1997), no.~2, 159--179.

\bibitem[BW21]{barthwenz}
D.~Barth and A.~Wenz, \emph{Computation of {Belyi} maps with prescribed
  ramification and applications in {Galois} theory}, J. Algebra \textbf{569}
  (2021), 616--642.

\bibitem[Che95]{chernousov}
V.~I. Chernousov, \emph{Galois cohomology and weak approximation for quotient
  varieties {${\bf A}^n/G$}}, Proc. Steklov Inst. Math. \textbf{208} (1995),
  296--308, translation from Tr.\ Mat.\ Inst.\ Steklova 208, 335--349 (1995).

\bibitem[CT00]{ctrc}
J.-L. Colliot-Th{\'e}l{\`e}ne, \emph{Rational connectedness and {Galois} covers
  of the projective line}, Ann. Math. (2) \textbf{151} (2000), no.~1, 359--373.

\bibitem[CT14]{ctbrnslng}
\bysame, \emph{Groupe de {Brauer} non ramifié de quotients par un groupe
  fini}, Proc. Am. Math. Soc. \textbf{142} (2014), no.~5, 1457--1469.

\bibitem[CTO89]{ctojanguren}
J.-L. Colliot-Th{\'e}l{\`e}ne and M.~Ojanguren, \emph{Vari{\'e}t{\'e}s
  unirationnelles non rationnelles~: au-del{\`a} de l'exemple d'{Artin} et
  {Mumford}}, Invent. math. \textbf{97} (1989), no.~1, 141--158.

\bibitem[CTPS16]{cps}
J.-L. Colliot-Th{\'e}l{\`e}ne, A.~P{\'a}l, and A.~N. Skorobogatov,
  \emph{Pathologies of the {Brauer}-{Manin} obstruction}, Math. Z. \textbf{282}
  (2016), no.~3-4, 799--817.

\bibitem[CTS87]{ctsandescent2}
J.-L. Colliot-Th{\'e}l{\`e}ne and J.-J. Sansuc, \emph{La descente sur les
  vari\'et\'es rationnelles. {II}}, Duke Math. J. \textbf{54} (1987), no.~2,
  375--492.

\bibitem[CTS07]{ct-sansuc-rationality}
J.-L. Colliot-Th\'{e}l\`ene and J.-J. Sansuc, \emph{The rationality problem for
  fields of invariants under linear algebraic groups (with special regards to
  the {B}rauer group)}, Algebraic groups and homogeneous spaces, Tata Inst.
  Fund. Res. Stud. Math., vol.~19, Tata Inst. Fund. Res., Mumbai, 2007,
  pp.~113--186.

\bibitem[CTS21]{ctskobook}
J.-L. Colliot-Th{\'e}l{\`e}ne and A.~N. Skorobogatov, \emph{The
  {Brauer}-{Grothendieck} group}, Ergeb. Math. Grenzgeb., 3. Folge, vol.~71,
  Springer, Cham, 2021.

\bibitem[D{\`e}b99]{debessurvey}
P.~D{\`e}bes, \emph{Arithmétique et espaces de modules de revêtements},
  Number theory in progress. Proceedings of the international conference
  organized by the Stefan Banach International Mathematical Center in honor of
  the 60th birthday of Andrzej Schinzel, Zakopane, Poland, June 30--July 9,
  1997. Volume 1: Diophantine problems and polynomials, Berlin: de Gruyter,
  1999, pp.~75--102.

\bibitem[Del73]{deligneweil1}
P.~Deligne, \emph{La conjecture de {Weil}. {I}}, Publ. Math. de l'I.H.\'E.S.
  \textbf{43} (1973), 273--307.

\bibitem[Dem10]{demarchebrnr}
C.~Demarche, \emph{Groupe de {B}rauer non ramifi\'e d'espaces homog\`enes \`a
  stabilisateurs finis}, Math. Ann. \textbf{346} (2010), no.~4, 949--968.

\bibitem[Dem21]{demeioramified}
J.~L. Demeio, \emph{Ramified descent and transcendental {B}rauer--{M}anin
  obstruction}, to appear in Algebra \& Number Theory,
  \href{http://arxiv.org/abs/2112.00843}{arXiv:2112.00843}, 2021.

\bibitem[DFV22]{dieulefaitfloritvila}
L.~Dieulefait, E.~Florit, and N.~Vila, \emph{Seven small simple groups not
  previously known to be {G}alois over {$\mathbf{Q}$}}, Mathematics \textbf{10}
  (2022), no.~12, 2048, 10 pages.

\bibitem[DLAN17]{dlan}
C.~Demarche, G.~Lucchini~Arteche, and D.~Neftin, \emph{The {G}runwald problem
  and approximation properties for homogeneous spaces}, Ann. Inst. Fourier
  (Grenoble) \textbf{67} (2017), no.~3, 1009--1033.

\bibitem[DR00]{dettweilerreiter}
M.~Dettweiler and S.~Reiter, \emph{An algorithm of {Katz} and its application
  to the inverse {Galois} problem}, J. Symb. Comput. \textbf{30} (2000), no.~6,
  761--798.

\bibitem[Eke90]{ekedahl}
T.~Ekedahl, \emph{An effective version of {H}ilbert's irreducibility theorem},
  S\'eminaire de Th\'eorie des Nombres, Paris 1988--1989, Progr. Math.,
  vol.~91, Birkh\"auser Boston, Boston, MA, 1990, pp.~241--249.

\bibitem[EM73]{endomiyata}
S.~Endo and T.~Miyata, \emph{Invariants of finite {Abelian} groups}, J. Math.
  Soc. Japan \textbf{25} (1973), 7--26.

\bibitem[Fis15]{fischer}
E.~Fischer, \emph{Die {Isomorphie} der {Invariantenk{\"o}rper} der endlichen
  {Abelschen} {Gruppen} linearer {Transformationen}}, Nachr. Ges. Wiss.
  G{\"o}ttingen, Math.-Phys. Kl. \textbf{1915} (1915), 77--80.

\bibitem[FLN22]{freiloughrannewton}
C.~Frei, D.~Loughran, and R.~Newton, \emph{Number fields with prescribed norms
  (with an appendix by {Y.~Harpaz} and {O.~Wittenberg})}, Comment. Math. Helv.
  \textbf{97} (2022), no.~1, 133--181.

\bibitem[Flo08]{florenceed}
M.~Florence, \emph{On the essential dimension of cyclic {{\(p\)}}-groups},
  Invent. math. \textbf{171} (2008), no.~1, 175--189.

\bibitem[Ful69]{fultonhurwitz}
W.~Fulton, \emph{Hurwitz schemes and irreducibility of moduli of algebraic
  curves}, Ann. Math. (2) \textbf{90} (1969), 542--575.

\bibitem[FV91]{friedvolkleinmoduli}
M.~D. Fried and H.~V{\"o}lklein, \emph{The inverse {Galois} problem and
  rational points on moduli spaces}, Math. Ann. \textbf{290} (1991), no.~4,
  771--800.

\bibitem[GMS03]{garibaldimerkurjevserre}
S.~Garibaldi, A.~Merkurjev, and J.-P. Serre, \emph{Cohomological invariants in
  {Galois} cohomology}, Univ. Lect. Ser., vol.~28, Providence, RI: American
  Mathematical Society (AMS), 2003.

\bibitem[Gro03]{sga1}
A.~Grothendieck (ed.), \emph{S{\'e}minaire de g{\'e}om{\'e}trie alg{\'e}brique
  du {Bois} {Marie} 1960-61~: {rev{\^e}tements} {\'e}tales et groupe
  fondamental ({SGA} 1)}, Documents Math{\'e}matiques, vol.~3, Soci{\'e}t{\'e}
  Math{\'e}matique de France, 2003.

\bibitem[Gru33]{grunwald}
W.~Grunwald, \emph{A general existence theorem for algebraic number fields}, J.
  reine angew. Math. \textbf{169} (1933), 103--107.

\bibitem[GS17]{gilleszamuelycsa}
P.~Gille and T.~Szamuely, \emph{Central simple algebras and {Galois}
  cohomology}, 2nd ed., Camb. Stud. Adv. Math., vol. 165, Cambridge: Cambridge
  University Press, 2017.

\bibitem[Har87]{harbaterinvgal}
D.~Harbater, \emph{Galois coverings of the arithmetic line}, Number theory. {A}
  {Seminar} held at the {Graduate} {School} and {University} {Center} of the
  {City} {University} of {New} {York} 1984-85, Lect. Notes Math., vol. 1240,
  Springer, Cham, 1987.

\bibitem[Har95]{harbatericm}
\bysame, \emph{Fundamental groups of curves in characteristic {{\(p\)}}},
  Proceedings of the international congress of mathematicians, ICM '94, August
  3-11, 1994, Z\"urich, Switzerland. Vol. I, Birkh{\"a}user, Basel, 1995,
  pp.~656--666.

\bibitem[Har07]{harariquelques}
D.~Harari, \emph{Quelques propri\'et\'es d'approximation reli\'ees \`a la
  cohomologie galoisienne d'un groupe alg\'ebrique fini}, Bull. Soc. Math.
  France \textbf{135} (2007), no.~4, 549--564.

\bibitem[Hil92]{hilbertorig}
D.~Hilbert, \emph{{\"U}ber die {Irreducibilit{\"a}t} ganzer rationaler
  {Functionen} mit ganzzahligen {Coefficienten}}, J. reine angew. Math.
  \textbf{110} (1892), 104--129.

\bibitem[Hos15]{hoshicomputer}
A.~Hoshi, \emph{On {Noether}'s problem for cyclic groups of prime order}, Proc.
  Japan Acad., Ser. A \textbf{91} (2015), no.~3, 39--44.

\bibitem[Hos20]{hoshisurvey}
\bysame, \emph{Noether's problem and rationality problem for multiplicative
  invariant fields: a survey}, RIMS K{\^o}ky{\^u}roku Bessatsu \textbf{B77}
  (2020), 29--53.

\bibitem[Hur91]{hurwitzorig}
A.~Hurwitz, \emph{{U}eber {R}iemann'sche {F}lächen mit gegebenen
  {V}erzweigungspunkten}, Math. Ann. \textbf{39} (1891), 1--61.

\bibitem[HW20]{hwzceh}
Y.~Harpaz and O.~Wittenberg, \emph{Z\'{e}ro-cycles sur les espaces homog\`enes
  et probl\`eme de {G}alois inverse}, J. Amer. Math. Soc. \textbf{33} (2020),
  no.~3, 775--805.

\bibitem[HW24]{hwsupersolvable}
\bysame, \emph{Supersolvable descent for rational points}, Algebra Number
  Theory \textbf{18} (2024), no.~4, 787--814.

\bibitem[Hä22]{haefner}
F.~Häfner, \emph{Braid orbits and the {M}athieu group ${M}_{23}$ as {G}alois
  group}, \href{http://arxiv.org/abs/2202.08222}{arXiv:2202.08222}, 2022.

\bibitem[JLY02]{jensenledetyui}
C.~U. Jensen, A.~Ledet, and N.~Yui, \emph{Generic polynomials. {Constructive}
  aspects of the inverse {Galois} problem}, Math. Sci. Res. Inst. Publ.,
  vol.~45, Cambridge: Cambridge University Press, 2002.

\bibitem[Jou83]{jouanoloubertini}
J.-P. Jouanolou, \emph{Th{\'e}or{\`e}mes de {Bertini} et applications}, Prog.
  Math., vol.~42, Birkh{\"a}user, Cham, 1983.

\bibitem[KN71]{krullneukirch}
W.~Krull and J.~Neukirch, \emph{Die {Struktur} der absoluten {Galoisgruppe}
  {\"u}ber dem {K{\"o}rper} {{\(\R(t)\)}}}, Math. Ann. \textbf{193} (1971),
  197--209.

\bibitem[Kol96]{kollarbook}
J.~Koll{\'a}r, \emph{Rational curves on algebraic varieties}, Ergeb. Math.
  Grenzgeb. (3), vol.~32, Springer-Verlag, Berlin, 1996.

\bibitem[Kol99]{kollarloc}
\bysame, \emph{Rationally connected varieties over local fields}, Ann. Math.
  (2) \textbf{150} (1999), no.~1, 357--367.

\bibitem[Kol00]{kollarfundamentalgroups}
\bysame, \emph{Fundamental groups of rationally connected varieties}, Mich.
  Math. J. \textbf{48} (2000), 359--368.

\bibitem[Kol03]{kollarfundamentalgroups2}
\bysame, \emph{Rationally connected varieties and fundamental groups}, Higher
  dimensional varieties and rational points. Lectures of the summer school and
  conference, Budapest, Hungary, September 3--21, 2001, Berlin: Springer;
  Budapest: J{\'a}nos Bolyai Mathematical Society, 2003, pp.~69--92.

\bibitem[LA19]{lucchiniunramifiedbrauer}
G.~Lucchini~Arteche, \emph{The unramified {B}rauer group of homogeneous spaces
  with finite stabilizer}, Trans. Amer. Math. Soc. \textbf{372} (2019), no.~8,
  5393--5408.

\bibitem[Len74]{lenstrainvent}
H.~W.~Jr. Lenstra, \emph{Rational functions invariant under a finite {Abelian}
  group}, Invent. math. \textbf{25} (1974), 299--325.

\bibitem[Len80]{lenstracyclic}
\bysame, \emph{Rational functions invariant under a cyclic group}, Proc.
  {Queen}'s {Number} {Theory} {Conf}. 1979, {Queen}'s {Pap}. {Pure} {Appl}.
  {Math}. 54, 91-99 (1980), 1980.

\bibitem[Mae89]{maeda}
T.~Maeda, \emph{Noether's problem for {{\(A_ 5\)}}}, J. Algebra \textbf{125}
  (1989), no.~2, 418--430.

\bibitem[Man71]{maninicm}
Yu.~I. Manin, \emph{Le groupe de {B}rauer-{G}rothendieck en g\'eom\'etrie
  diophantienne}, Actes du {C}ongr\`es {I}nternational des {M}ath\'ematiciens
  ({N}ice, 1970), {T}ome 1, Gauthier-Villars, Paris, 1971, pp.~401--411.

\bibitem[Mat87]{matzatkonstruktive}
B.~H. Matzat, \emph{Konstruktive {Galoistheorie}}, Lect. Notes Math., vol.
  1284, Springer, Cham, 1987.

\bibitem[MB01]{moretbaillyconstruction}
L.~Moret-Bailly, \emph{Construction de revêtements de courbes pointées}, J.
  Algebra \textbf{240} (2001), no.~2, 505--534.

\bibitem[Mer13]{merkurjevsurveyed}
A.~S. Merkurjev, \emph{Essential dimension: a survey}, Transform. Groups
  \textbf{18} (2013), no.~2, 415--481.

\bibitem[Mer17]{merkurjevedbis}
\bysame, \emph{Essential dimension}, Bull. Am. Math. Soc., New Ser. \textbf{54}
  (2017), no.~4, 635--661.

\bibitem[Mes90]{mestreantilde}
J.-F. Mestre, \emph{Extensions r{\'e}guli{\`e}res de {{\({\mathbb{Q}}(T)\)}} de
  groupe de galois {{\(\tilde A_ n\)}}}, J. Algebra \textbf{131} (1990), no.~2,
  483--495.

\bibitem[Mes94]{mestresl2f7m12tilde}
\bysame, \emph{Construction of regular extensions of {{\(\mathbb{Q}(t)\)}} with
  galois groups {{\(\text{SL}_ 2 (\mathbb{F}_ 7)\)}} and
  {{\(\widetilde{M}_{12}\)}}}, C. R. Acad. Sci., Paris, S{\'e}r. I \textbf{319}
  (1994), no.~8, 781--782.

\bibitem[Mes98]{mestre6a66a7}
\bysame, \emph{{{\(\mathbb{Q}\)}}-regular extensions of {{\({\mathbb{Q}}(t)\)}}
  with galois groups {{\(6. A_6\)}} and {{\(6. A_7\)}}}, Isr. J. Math.
  \textbf{107} (1998), 333--341.

\bibitem[Mes05]{mestrepsl2f7}
\bysame, \emph{Correspondances compatibles avec une relation binaire,
  rel\`evement d'extensions de groupe de {G}alois $\mathrm{L}_3(2)$ et
  probl\`eme de {N}oether pour $\mathrm{L}_3(2)$},
  \href{https://arxiv.org/abs/math/0402187}{arXiv:math/0402187}, 2005.

\bibitem[MFK94]{mumfordgit}
D.~Mumford, J.~Fogarty, and F.~Kirwan, \emph{Geometric invariant theory}, 3rd
  enl. ed., Ergebnisse der Mathematik und ihrer Grenzgebiete, 3.~Folge,
  vol.~34, Berlin: Springer-Verlag, 1994.

\bibitem[Mil20]{milneCFT}
J.~S. Milne, \emph{Class field theory}, 2020, course notes, version 4.03,
  available from the author's webpage at
  \href{https://www.jmilne.org/math/CourseNotes/CFT.pdf}{https://www.jmilne.org/math/CourseNotes/CFT.pdf},
  pp.~287+viii.

\bibitem[MM18]{mallematzat}
G.~Malle and B.~H. Matzat, \emph{Inverse {Galois} theory}, 2nd ed., Springer
  Monogr. Math., Berlin: Springer, 2018.

\bibitem[Mum08]{mumford}
D.~Mumford, \emph{Abelian varieties}, Tata Institute of Fundamental Research
  Studies in Mathematics, No.~5, Hindustan Book Agency, New Delhi, 2008, with
  appendices by C. P. Ramanujam and Yu. Manin; corrected reprint of the second
  (1974) edition.

\bibitem[Neu79]{neukirch-solvable}
J.~Neukirch, \emph{On solvable number fields}, Invent. math. \textbf{53}
  (1979), no.~2, 135--164.

\bibitem[Ngu24]{nguyendescent}
M.~L. Nguy{\~{\^e}n}, \emph{On the descent conjecture for rational points and
  zero-cycles}, to appear in {Journal de l’Institut de Mathématiques de
  Jussieu}, \href{https://arxiv.org/abs/2305.13228}{arXiv:2305.13228}, 2024.

\bibitem[NSW08]{neukirchschmidtwingberg}
J.~Neukirch, A.~Schmidt, and K.~Wingberg, \emph{Cohomology of number fields},
  second ed., Grundlehren der Mathematischen Wissenschaften, vol. 323,
  Springer-Verlag, Berlin, 2008.

\bibitem[Pey08]{peyrenoether}
E.~Peyre, \emph{Unramified cohomology of degree 3 and {Noether}'s problem},
  Invent. math. \textbf{171} (2008), no.~1, 191--225.

\bibitem[Pla17]{plans}
B.~Plans, \emph{On {Noether}'s rationality problem for cyclic groups over
  {{\(\mathbb {Q}\)}}}, Proc. Am. Math. Soc. \textbf{145} (2017), no.~6,
  2407--2409.

\bibitem[Poo17]{poonenqpoints}
B.~Poonen, \emph{Rational points on varieties}, Grad. Stud. Math., vol. 186,
  Providence, RI: American Mathematical Society (AMS), 2017.

\bibitem[Pop96]{poplarge}
F.~Pop, \emph{Embedding problems over large fields}, Ann. Math. (2)
  \textbf{144} (1996), no.~1, 1--34.

\bibitem[PV04]{poonenvoloch}
B.~Poonen and J.~F. Voloch, \emph{Random {D}iophantine equations}, Arithmetic
  of higher-dimensional algebraic varieties ({P}alo {A}lto, {CA}, 2002), Progr.
  Math., vol. 226, Birkh\"auser Boston, Boston, MA, 2004, with appendices by
  J.-L. Colliot-Th{\'e}l{\`e}ne and N.~M.~Katz, pp.~175--184.

\bibitem[Rei11]{reichsteinicm}
Z.~Reichstein, \emph{Essential dimension}, Proceedings of the international
  congress of mathematicians (ICM 2010), Hyderabad, India, August 19--27, 2010.
  Vol.~II: Invited lectures, Hackensack, NJ: World Scientific; New Delhi:
  Hindustan Book Agency, 2011, pp.~162--188.

\bibitem[Rei21]{reichstein13}
\bysame, \emph{From {Hilbert}'s 13th problem to essential dimension and back},
  Eur. Math. Soc. Mag. \textbf{122} (2021), 4--15.

\bibitem[RW06]{romagnywewers}
M.~Romagny and S.~Wewers, \emph{Hurwitz spaces}, Groupes de Galois
  arithm\'etique et diff\'erentiels, Soci{\'e}t{\'e} Math{\'e}matique de
  France, 2006, pp.~313--341.

\bibitem[Sal82]{saltmangeneric}
D.~J. Saltman, \emph{Generic {Galois} extensions and problems in field theory},
  Adv. Math. \textbf{43} (1982), 250--283.

\bibitem[Sal84]{saltmannoether}
\bysame, \emph{Noether's problem over an algebraically closed field}, Invent.
  math. \textbf{77} (1984), 71--84.

\bibitem[Sal85]{saltmangenericmatrices}
\bysame, \emph{The {Brauer} group and the center of generic matrices}, J.
  Algebra \textbf{97} (1985), 53--67.

\bibitem[San81]{sansuclinear}
J.-J. Sansuc, \emph{Groupe de {B}rauer et arithm\'etique des groupes
  alg\'ebriques lin\'eaires sur un corps de nombres}, J.~reine angew. Math.
  \textbf{327} (1981), 12--80.

\bibitem[Ser97]{serremw}
J.-P. Serre, \emph{Lectures on the {Mordell}-{Weil} theorem}, 3rd ed., Aspects
  Math., vol. E15, Wiesbaden: Vieweg, 1997.

\bibitem[Ser07]{serretopics}
\bysame, \emph{Topics in {Galois} theory}, 2nd ed., Res. Notes Math., vol.~1,
  Wellesley, MA: A K Peters, 2007.

\bibitem[Shi74]{shih}
K.-y. Shih, \emph{On the construction of {Galois} extensions of function fields
  and number fields}, Math. Ann. \textbf{207} (1974), 99--120.

\bibitem[Sko01]{skobook}
A.~N. Skorobogatov, \emph{Torsors and rational points}, Cambridge Tracts in
  Mathematics, vol. 144, Cambridge University Press, Cambridge, 2001.

\bibitem[Spe19]{speiser}
A.~Speiser, \emph{Zahlentheoretische {S{\"a}tze} aus der {Gruppentheorie}},
  Math. Z. \textbf{5} (1919), 1--6.

\bibitem[Swa69]{swan47}
R.~G. Swan, \emph{Invariant rational functions and a problem of {Steenrod}},
  Invent. math. \textbf{7} (1969), 148--158.

\bibitem[Swa83]{swansurvey}
\bysame, \emph{Noether's problem in {Galois} theory}, Emmy {Noether} in {Bryn}
  {Mawr}, {Proc}. {Symp}., {Bryn} {Mawr}/{USA} 1982, 21--40 (1983), 1983.

\bibitem[Sza09]{szamuelygaloisgroups}
T.~Szamuely, \emph{Galois groups and fundamental groups}, Camb. Stud. Adv.
  Math., vol. 117, Cambridge: Cambridge University Press, 2009.

\bibitem[V{\"o}l96]{volklein}
H.~V{\"o}lklein, \emph{Groups as {Galois} groups: an introduction}, Camb. Stud.
  Adv. Math., vol.~53, Cambridge: Cambridge Univ. Press, 1996.

\bibitem[V{\"o}l01]{volkleinbc}
\bysame, \emph{The braid group and linear rigidity}, Geom. Dedicata \textbf{84}
  (2001), no.~1-3, 135--150.

\bibitem[Vos70]{voskbirlin}
V.~E. Voskresenski{\u{\i}}, \emph{Birational properties of linear algebraic
  groups}, Izv. Akad. Nauk SSSR, Ser. Mat. \textbf{34} (1970), 3--19.

\bibitem[Vos71]{voskprime}
\bysame, \emph{Rationality of certain algebraic tori}, Izv. Akad. Nauk SSSR,
  Ser. Mat. \textbf{35} (1971), 1037--1046.

\bibitem[Vos73]{voskfieldofinvariants}
V.~E. Voskresenskii, \emph{Fields of invariants for {Abelian} groups}, Usp.
  Mat. Nauk \textbf{28} (1973), no.~4 (172), 77--102.

\bibitem[Vos98]{voskbook}
V.~E. Voskresenski{\u\i}, \emph{Algebraic groups and their birational
  invariants}, Translations of Mathematical Monographs, vol. 179, American
  Mathematical Society, Providence, RI, 1998.

\bibitem[Wan50]{wanggrunwald}
S.~Wang, \emph{On {Grunwald}'s theorem}, Ann. Math. (2) \textbf{51} (1950),
  471--484.

\bibitem[Wew98]{wewersphd}
S.~Wewers, \emph{Construction of {Hurwitz} spaces}, Ph.D. thesis, Essen, 1998.

\bibitem[Wit18]{wittenbergslc}
O.~Wittenberg, \emph{Rational points and zero-cycles on rationally connected
  varieties over number fields}, Algebraic geometry: {S}alt {L}ake {C}ity 2015,
  Proc. Sympos. Pure Math., vol.~97, Amer. Math. Soc., Providence, RI, 2018,
  pp.~597--635.

\bibitem[Zyw13]{zywinasmall}
D.~Zywina, \emph{Inverse {G}alois problem for small simple groups}, unpublished
  note, available from the author's webpage at
  \href{http://pi.math.cornell.edu/~zywina/papers/smallGalois.pdf}{http://pi.math.cornell.edu/\~{}zywina/papers/smallGalois.pdf},
  2013.

\bibitem[Zyw15]{zywina}
\bysame, \emph{The inverse {Galois} problem for {{\(\mathrm{PSL}_2(\mathbb
  {F}_p)\)}}}, Duke Math. J. \textbf{164} (2015), no.~12, 2253--2292.

\end{thebibliography}
\end{document}